\documentclass[11pt,a4paper,oneside]{amsart}
\usepackage[utf8]{inputenc}
\usepackage{a4wide}

\usepackage{amsmath,amssymb,bbm}
\usepackage{aliascnt}

\usepackage{csquotes}

\usepackage[all,cmtip]{xy}

\usepackage[hidelinks]{hyperref}
\usepackage{amsthm}

\newcommand*{\mc}[1]{\mathcal{#1}}
\newcommand*{\opname}[1]{\operatorname{#1}}
\newcommand*{\cat}[1]{\textnormal{\underline{#1}}}

\newcommand*{\GL}{\opname{GL}}
\newcommand*{\SL}{\opname{SL}}
\newcommand*{\PGL}{\opname{PGL}}
\renewcommand*{\phi}{\varphi}

\newcommand*{\NN}{\mathbbm{N}}
\newcommand*{\ZZ}{\mathbbm{Z}}
\newcommand*{\CC}{\mathbbm{C}}
\newcommand*{\QQ}{\mathbbm{Q}}
\newcommand*{\EE}{\mathbbm{E}}
\newcommand*{\PP}{\mathbbm{P}}

\newcommand*{\Jac}{\opname{Jac}}
\newcommand*{\ka}{\kappa}
\newcommand*{\la}{\lambda}
\newcommand*{\ga}{\gamma}

\newcommand*{\rk}{\opname{rk}}
\newcommand*{\pr}{\opname{pr}}
\newcommand*{\id}{\opname{id}}
\newcommand*{\tup}[1]{\underline{#1}}
\newcommand*{\gr}{{\opname{gr}}}
\newcommand*{\len}{\opname{l}}
\newcommand*{\Hom}{\opname{Hom}}
\newcommand*{\End}{\opname{End}}
\newcommand*{\ev}{\opname{ev}}
\newcommand*{\Fl}{\opname{Fl}}

\newcommand*{\Spec}{\opname{Spec}}

\newcommand*{\Gies}{\opname{Gies}}
\newcommand*{\gies}{\opname{gies}}
\newcommand*{\Quot}{\opname{Quot}}

\newcounter{internal}[section]
 
\newaliascnt{intthm}{internal} 
\newaliascnt{intprop}{internal} 
\newaliascnt{intcor}{internal} 
 
\newaliascnt{intlemma}{internal} 
\newaliascnt{intdef}{internal} 
\newaliascnt{intex}{internal} 
\newaliascnt{intrem}{internal}

\theoremstyle{plain}
\newtheorem{thm}[intthm]{Theorem}
\newtheorem{prop}[intprop]{Proposition}
\newtheorem{cor}[intcor]{Corolary}
\newtheorem{lemma}[intlemma]{Lemma}
\theoremstyle{definition}
\newtheorem{mydef}[intdef]{Definition}
\theoremstyle{remark}
\newtheorem{rem}[intrem]{Remark}

\begin{document}
\title[Moduli of Decorated Swamps]
    {Moduli of Decorated Swamps on a Smooth Projective Curve}

\author{Nikolai Beck}
\email{Nikolai.Beck@b-tu.de}
\address{
  Mathematisches Institut\\
  BTU Cottbus--Senftenberg\\
  PF 101344\\
  03013 Cottbus\\
  Germany}

%\classification{14H60, 14D20}

\begin{abstract}
In order to unify the construction of the moduli space of vector bundles with different types of global decorations, such as Higgs bundles, framed vector bundles and conic bundles, A. Schmitt introduced the concept of a swamp. In this work, we consider vector bundles with both a global and a local decoration over a fixed point of the base. This generalizes the notion of parabolic vector bundles, vector bundles with a level structure and parabolic Higgs bundles. We introduce a notion of stability and construct the coarse moduli space for these objects as the GIT-quotient of a parameter space. In the case of parabolic vector bundles and vector bundles with a level structure our stability concept reproduces the known ones. Thus, our work unifies the construction of their moduli spaces. 
\end{abstract}

\maketitle

%%%%%%%%%%%%%%%%%%%%%%%%%%%%%%%%%%%%%%%
\section{Introduction}
Let $X$ be a smooth projective curve over the field of complex numbers and fix a point $x_0$ of $X$. There are different reasons to equip vector bundles on $X$ with certain global and local decorations: a parabolic Higgs bundle is a vector bundle $E$ together with a Higgs field $\phi:E\to E\otimes \omega_X$, where $\omega_X$ denotes the canonical sheaf, and a weighted flag over $x_0$ which is $\phi$-invariant. There is a notion of stability and an extension of the Narasimhan--Seshadri theorem tells us that a parabolic Higgs bundle is stable if and only if it arises from an irreducible representation of $\pi_1(X\setminus\{x_0\})$ with fixed monodromy \cite{simpson1990}.

Another example of a local decoration on a vector bundle $E$ is a level structure, i.e., a complete homomorphism $f:E_{|\{x_0\}}\Rightarrow \CC^{\rk(E)}$. Ng\^o Dac \cite{Ngo2007} defined a stability condition for these objects and applied Geometric invariant theory (GIT) in order to obtain a compactification of the stack of shtukas, which plays an important role in the Langlands program. 

A generalized version of a vector bundle with a global decoration is a swamp: fix a representation $\rho:\GL(r)\to \GL(V)$. Given a vector bundle $E$ of rank $r$ we can construct its associated vector bundle $E_\rho$ with typical fiber $V$. Then, a swamp is a vector bundle of rank $r$ together with a line bundle $L$ and a non-trivial homomorphism $\phi:E_\rho\to L$. A. Schmitt introduced a stability condition for swamps and constructed their coarse moduli space \cite{schmitt08}.

In this work we generalize the concept of local decorations in the same spirit: we fix two representations $\rho$ and $\sigma$ of $\GL(r)$. We define a \emph{decorated swamp} as a swamp $(E,L,\phi)$ together with a point $s\in E^{\vee}_\sigma$ over $x_0$. We define a parameter dependent notion of stability for these objects and construct a parameter space with a locally universal family and a compatible group action. Then we define an equivariant morphism from this space into a projective Gieseker space. A careful calculation shows that a decorated swamp is stable if and only if the corresponding point in the Gieseker space is GIT-stable with respect to a certain linearization. Finally, the coarse moduli space is obtained as the GIT-quotient of the parameter space.

In the  case that $\rho$ is trivial we call a decorated swamp a decorated vector bundle. For certain choices of $\sigma$ such an object describes a parabolic vector bundle or a vector bundle with a level structure. We show that in these cases our concept of stability coincides with the usual ones.

While the category of decorated swamps also describes parabolic Higgs bundles, we cannot expect the stability conditions to coincide since in contrast to the stability of parabolic Higgs bundles our definition of stability is parameter dependent. As was the case for swamps and Higgs bundles, we expect to find the known stability condition in the form of asymptotic stability. This is planned for a future publication.

\subsection*{Notation and Conventions}
In this work we will identify a geometric vector bundle $E$ with its sheaf of sections. If $F$ is a subsheaf of $E$, the subbundle generically generated by $F$ is
\[
 \ker(F\to (E/F)/T) \,,
\]
where $T$ is the torsion subsheaf of $E/F$. We denote by $\PP(E)$ the hyperplane bundle $\opname{Proj}(\opname{Sym}^*E)$.

\subsection*{Acknowledgement}
These results presented here are part of the authors PhD thesis \cite{Beck2014} supervised by Alexander Schmitt, whom the author would like to thank.

%%%%%%%%%%%%%%%%%%%%%%%%%%%%%%%%%%%%%%%
\section{Preliminaries}
In order to fix notation and to recall some specific details which are important for our particular application, we gather some basic facts about Geometric invariant theory and moduli spaces. We also recall some properties of homogeneous representations.

\subsection{Geometric Invariant Theory}
\label{subsec:GIT}
Let $G$ be a linear reductive group and let $\rho:G\times X\to X$ be an action on a scheme $X$.
 A \emph{good quotient} is an affine $G$-invariant morphism $\pi:X \to Y$, such that
 \begin{enumerate}
  \item For any open subset $U\subset Y$ the map $\pi^{\#}:\mc{O}_Y(U)\to \mc{O}_X(\pi^{-1}(U))$ is an isomorphism.
  \item For any closed invariant subset $Z\subset X$ the image $\pi(Z)$ is closed in $Y$.
  \item If $Z_1,Z_2\subset X$ are disjoint closed invariant subset, then $\pi(Z_1)\cap\pi(Z_2)=\varnothing$.
 \end{enumerate}
If the preimage $\pi^{-1}(y)$ of any point $y\in Y$ contains only one orbit, then $Y$ is a \emph{geometric quotient}. A good quotient is in fact a categorical quotient. If $X$ is affine with coordinate ring $R$, then $\Spec(R^G)$ is a good quotient. For non-affine schemes we first note:
\begin{prop}[\protect{\cite[5.1. Lemma]{Ramanathan1996b}, \cite[Lemma 4.6]{Gieseker1977}}]
 \label{prop:affine_morphism->quotient}
 Let $f:X\to X'$ be an affine $G$-invariant morphism and suppose the good quotient $\pi:X'\to X'/\!\!/G$ exists.
 \begin{enumerate}
  \item The good quotient $\pi:X\to X/\!\!/G$ exists.
  \item If $f$ is proper and $\pi'$ is a universal good quotient, then the induced map $\bar{f}:X/\!\!/G\to X'/\!\!/G$ is also proper.
  \item If $f$ is injective and $\pi'$ is a geometric quotient, then $\pi$ is also a geometric quotient.
 \end{enumerate}
\end{prop}

In general, one needs some additional structure. A \emph{linearization} of $\rho$ in a line bundle $L$ on $X$ is an action $\bar{\rho}: G\times L \to L$ with $\pr_X\circ \bar{\rho}=\rho\circ (\id_G\times \pr_X)$, such that the map $L_x \to L_{\rho(g,x)}$ is linear for all $g\in G$. Given such a linearization a point $x\in X$ is called
 \begin{itemize}
  \item \emph{semistable} if there is an $r\in \NN_{>0}$ and a $G$-invariant section $s$ in $L^{\otimes r}$  with $s(x)\neq 0$ such that $X_{s}=X\setminus V(s)$ is affine,
  \item \emph{polystable} if $x$ is semistable and the orbit of $x$ is closed in $X_s$,
  \item \emph{stable} if there is $r\in \NN_{>0}$ and a $G$-invariant $s\in H^0(X,L^{\otimes r})$ such that $X_s$ is affine, $x\in X_s$, $\dim(G\cdot x)=\dim(G)$ and every orbit in $X_s$ is closed.
 \end{itemize}
We denote the open set of (semi-)stable points by $X^{\textnormal{(s)s}}$. 

\begin{prop}[\protect{Mumford, \cite[Theorem 1.10]{GIT}}]
 \label{prop:Mumford}
 Let $X$ be a scheme with an action of a linear reductive group $G$ and a linearization in a line bundle $L$.
 \begin{enumerate}
  \item The quasi-projective good quotient $\pi:X^{\textnormal{ss}}\to X/\!\!/G$ exists. If $X$ is projective and $L$ is ample, then $X/\!\!/G$ is also projective.
  \item There is an open subset $U\subset X/\!\!/G$ such that $\pi^{-1}(U)=X^{\textnormal{s}}$ and $\pi_{|X^{\textnormal{s}}}:X^{\textnormal{s}}\to U$ is a geometric quotient.
 \end{enumerate}
\end{prop}

Suppose that $X$ is projective. Let $\la:\CC^*\to G$ be a one-parameter subgroup and $x\in X$ a point. The limit point $x_\infty:=\lim_{t\to\infty} \la(t)\cdot x$ is a fixed point for the $\CC^*$ action. The action on the fiber $L_{x_\infty}$ is determined by $t\cdot l=t^{\gamma}l$ for some $\ga\in \ZZ$. One defines
\[
 \mu_{\rho}(\la,x):=-\ga\,.
\]
\begin{prop}[Hilbert--Mumford, \protect{\cite[Theorem 2.1]{GIT}}] 
\label{Hilbert-Mumford-Krit}
A point $x\in X$ is \textup{(}semi-\textup{)}stable with respect to  $\rho$ if and only if any non-trivial one-parameter subgroup $\lambda: \CC^*\to G$ satisfies
 \[
  \mu_\rho(\la,x)(\ge)0\,.
 \]
\end{prop}

Suppose $X=\PP(V)$ for a vector space $V$ and $G\subset \SL(V)$. Given a one-parameter subgroup $\la$ of $\SL(V)$ we find a splitting $V=\bigoplus_{i=1}^{k+1} V^i$ and weights $\gamma_1< \ldots<\gamma_{k+1}$ such that $\la(t)\cdot v=t^{\ga_i}v$ for $v\in V^i$. For a point $x\in \PP(V)$ represented by $f\in V^\vee$ one finds
\[
 \mu(\la,x)= -\min\{\ga_i\,|\,1\le i\le k+1: f_{|V^i}\neq 0\}\,.
\]

Consider the subspaces $V_j:=\bigoplus_{i=1}^j V^i$ and the positive rational numbers
\[
  \alpha_j:=\frac{\ga_{j+1}-\ga_j}{\dim(V)}\,, \qquad i=1,\ldots,k\,.
\]
The pair $(V_\bullet,\tup{\alpha})$ is the associated \emph{weighted flag} of $\la$ and $\len(V_\bullet):=k$ the \emph{length} of $V_\bullet$. Conversely, given the weighted flag $(V_\bullet,\tup{\alpha})$ one defines $\tup{\ga}(V_\bullet,\tup{\alpha})$ by
\[
 \ga_i(V_\bullet,\tup{\alpha}):=\sum_{j=1}^{\len(V_\bullet)} \alpha_j \dim(V_j)  -\sum_{j=i}^{\len(V_\bullet)} \alpha_j
\dim(V)\,, \qquad i=1,\ldots,\len(V_\bullet)+1
\]
There is a number $m\in \NN$ such that $\tup{\ga}:=\tup{\ga}(V_\bullet,m\tup{\alpha})$ is integral. The choice of a splitting of $V$ that is compatible with $V_\bullet$ then determines a one-parameter subgroup $\la$ with associated weighted flag $(V_\bullet,m\tup{\alpha})$. One can show that for $x\in \PP(V)$ the value of $\mu(\la,x)$ does not depend on the choice of $\la$ (see \cite[Proposition 2.7]{GIT}). Thus the function $\mu(V_\bullet,\tup{\alpha},x):=\mu(\la,x)/m$ is well defined.

\subsection{Moduli Spaces}
\label{subsec:moduli_spaces}
Let $\mc{A}$ be a set of objects. A \emph{moduli problem} for $\mc{A}$ consists of a category $\mc{F}$ fibered over $\cat{Sch}/\CC$ and an equivalence relation $\sim$ on $\mc{F}$ which is compatible with pullbacks, such that $\mc{F}(\id_\CC)/\!\sim\; =\mc{A}/\!\cong$. An object $\xi\in \mc{F}(S)$ is called a \emph{family of objects parameterized by $S$}. These data define the \emph{moduli functor}
\begin{align*}
 \cat{Sch}/\CC &\to \cat{Sets}\\
  S &\mapsto \mc{F}(S)/\!\sim\,.
\end{align*}

A scheme $M$ is a \emph{fine moduli space} if it represents the moduli functor. By the Yoneda-lemma this is equivalent to the existence of a universal family on $M$. Let $\sim_S$ be an equivalence relation on $\mc{A}$.
 A scheme $M$ is a \emph{coarse moduli space} if it corepresents (\cite[Definition 2.2.1]{HuyLehn}) the moduli functor and there is a natural bijection $ \mc{A}/\!\sim_S\; \to M(\CC)$.
 
A family $\xi$ parametrized by a scheme $P$ is said to satisfy the \emph{local universal property}, if for any family $\xi'$ parametrized by a scheme $S$ and any point $s\in S$ there is a neighborhood $U$ of $s$ and a morphism $f:U\to P$ such that $\xi'_{|U}\sim f^*\xi$. Suppose further that there is an action $\rho:G\times P\to P$ with the property that for any two morphisms $f_1,f_2:U\to P$ we have $f_1^*\xi\sim f_2^*\xi$ if and only if there exists a morphism $g:U\to G$ such that $g\cdot f_1=f_2$. Then, a scheme $M$ is a coarse moduli space of isomorphism classes if and only if $M$ is a geometric quotient of $P$.

If $M$ is the good quotient of $P$, but not a geometric quotient, one defines $\sim_S$ in the following way: two objects $E$ and $E'$ are \emph{S-equivalent} if there are points $p$, $p'\in P$ such that $\xi_{p}\cong E$, $\xi_{p'}\cong E'$ and $\overline{G\cdot p}\cap \overline{G\cdot p'}\neq \varnothing$.

\begin{prop} \label{prop:good_quot_is_coarse_moduli}
 The scheme $M$ is a coarse moduli space of S-equivalence classes if and only if $M$ is a good quotient of $P$.
\end{prop}
\begin{rem}
 Richardsons proof of the Hilbert--Mumford criterion given in \cite{Birkes} shows that S-equivalence is generated by $\xi_p\sim_S \xi_{p'}$ for any $p\in P$ and $p'=\lim_{t\to\infty}\rho(\la(t),p)$ with $\mu_{\rho}(\la,p)=0$.
\end{rem}

\subsection{Homogeneous Representations}
A representation $\rho:\GL(V)\to \GL(W)$ is called \emph{homogeneous of degree $\deg(\rho)=\ga$} if $\rho(c\id_V)=c^\ga \id_W$ for all $c\in \CC^*$. If $\rho$ is homogeneous, then by \cite[Corollary 1.1.5.4]{schmitt08} there are non-negative integers $a,b,c$ such that $W$ is a subrepresentation of
\begin{equation} \label{eq:V_abc}
 V_{a,b,c}:=\left(V^{\otimes a}\right)^{\oplus b}\otimes \left(\bigwedge^{\dim(V)} V \right)^{\otimes -c}.
\end{equation}

For a point $[f]\in \PP(V_{a,b,c})$ and a weighted flag $(V_\bullet,\tup{\alpha})$ one finds
\begin{equation} \label{eq:stability_V_abc}
 \mu(V_\bullet,\tup{\alpha},[f])= -\min\left\{\sum_{j=1}^{\len(V_\bullet)}\alpha_j\big(a\dim(V_j)-\dim(V)\nu_j(\tup{i})\big) \biggm| \tup{i}\in I\,:\, f_{V^{\otimes\tup{i}}}\neq 0\right\} 
\end{equation}
where $I:=\{1,\ldots,\len(V_\bullet)+1\}^a$, $V^{\otimes \tup{i}}:=\bigoplus_{j=1}^a V_{i_j}$ for $\tup{i}\in I$ and $\nu_j(\tup{i}):=\#\{i\in \tup{i}|i\le j\}$.
In particular we have the estimate (see \cite[Lemma 1.5.1.41]{schmitt08})
\begin{equation} \label{eq:estimate_V_abc}
   \mu(V_\bullet, \tup{\alpha}+\tup{\beta},x) \ge \mu(V_\bullet,\tup{\alpha},x) - a\sum_{i=1}^{\len(V_\bullet)} \beta_j \dim(V_j) \,.
\end{equation}

\subsection{Associated Vector Bundles and Weighted Flags}
Let $E$ be a vector bundle of rank $r$ and $\rho:\GL(r)\to \GL(V)$ a representation. We denote by $P(E)$ the frame bundle $\opname{Iso}(\mc{O}_X^{\oplus r},E)$ and by $E_\rho$ the associated bundle $P(E)\times^\rho V$.

\begin{mydef} \label{def:weighted_flag}
 A \emph{weighted flag} of a vector bundle $E$ is pair $(E_\bullet,\tup{\alpha})$ consisting of a flag
 \[
  \{0\}=E_0\subset E_1\subset \cdots \subset E_k\subset E_{k+1}=E
 \]
of length $\len(E_\bullet)=k$ and a $k$-tuple $\tup{\alpha}$ of positive rational numbers.
\end{mydef}

Suppose that $\rho$ is homogeneous. Given a weighted flag $(E_\bullet,\tup{\alpha})$ we choose a flag $W_\bullet$ of $\CC^r$ of length $\len(W_\bullet)=\len(E_\bullet)$ with $\dim(W_i)=\rk(E_i)$ for $1\le i \le \len(E_\bullet)$. If $m$ is an integer such that $m\tup{\alpha}$ is integral, then, as described at the end of \autoref{subsec:GIT}, there is a one-parameter subgroup $\la$ of $\SL(r)$ with associated weighted flag $(W_\bullet,m\tup{\alpha})$. The one-parameter subgroup $\rho\circ \la$ of $\SL(V)$ now determines a weighted flag $(V_\bullet,\tup{\beta})$ of $V$.

One can find an open subset $U\subset X$ and a trivialization $\psi:E_{|U}\to U\times \CC^r$ such that $\psi(E_{i|U})=U\times W_i$ for $1\le i\le \len(E_\bullet)$. The trivialization $\psi$ induces an associated trivialization $\psi_\rho:E_{\rho|U}\to U\times V$. The flag of $E_{\rho|U}$ defined by $\psi_\rho^{-1}(U\times V_i)$ for $1\le i \le \len(V_\bullet)$ extends uniquely to a flag $F_\bullet$ of $E_\rho$.
\begin{mydef}
We define the \emph{associated weighted flag $(E_{\rho,\bullet},\tup{\alpha}_\rho)$ of $E_\rho$ induced by $(E_\bullet,\tup{\alpha})$ and $\rho$} by $E_{\rho,\bullet}:=F_\bullet$ and $\tup{\alpha}_\rho:=(1/m)\tup{\beta}$.
\end{mydef}

%%%%%%%%%%%%%%%%%%%%%%%%%%%%%%%%%%%%%%%
\section{Stable Decorated Swamps}
In this section we introduce the objects that we want to classify, define notions of stability,  S-equivalence and parameterized families and finally state the existence of the coarse moduli space.

\subsection{Decorated Swamps}
Let $X$ be a smooth projective curve over $\CC$ of genus $g$.
Fix an integer $d$, a natural number $r$ and two homogeneous representations $\rho:\GL(r)\to \GL(V_1)$ and $\sigma:\GL(r)\to \GL(V_2)$.
\begin{mydef}
 A \emph{decorated swamp} of \emph{type $(d,l)$} is a tuple $(E,L,\phi,s)$, where $E$ is vector bundle of rank $r$ and degree $d$, $L$ is a line bundle of degree $l$, $\phi:E_\rho\to L$ is a non-trivial homomorphism and $s$ is a point in $E^{\vee}_{\sigma|\{x_0\}}$. 

Two decorated swamps $(E,L,\phi,s)$ and $(E',L',\phi',s')$ are considered \emph{isomorphic} if there are isomorphisms $f:E\to E'$, $\psi:L\to L'$ and a number $c\in \CC^*$ with $\phi'\circ f_\rho=\psi\circ \phi$ and $ s\circ f_{\sigma|\{x_0\}}=c\cdot s$. Here $f_\rho:E_\rho\to E'_\rho$ and $f_\sigma:E_\sigma\to E'_\sigma$ are the isomorphisms induced by $f$.
\end{mydef}

\subsection{Stability}
Let $(E,L,\phi,s)$ be a decorated swamp. Stability is going to be tested for weighted flags $(E_\bullet,\tup{\alpha})$ of $E$. We define 
\[
 M(E_\bullet,\tup{\alpha}):=\sum_{j=1}^{\len(V_\bullet)}\alpha_j\left(\deg(E)\rk(E_j)-\deg(E_j)\rk(E) \right)\,.
\]

Denote by $\eta$ be the generic point of $X$ and by $K$ the function field of $X$.
The representation $\rho$ and the flag $(E_\bullet,\tup{\alpha})$ induce an associated weighted flag $(E_{\rho,\bullet},\tup{\alpha}_\rho)$ of $E_\rho$. Restricting to the generic point $\eta$ yields a flag $\EE_{\rho,\bullet}$ of the $K$-vector space $\EE_\rho:=E_{\rho|\eta}$. Finally, the homomorphism $\phi$ defines a point $[\phi_\eta]\in \PP(\EE_\rho)$. We set
\[
 \mu_1(E_\bullet,\tup{\alpha},\phi):=\mu(\EE_{\rho,\bullet},\tup{\alpha}_\rho,[\phi_{\eta}])\,.
\]
\begin{rem}
 This is the same function as in \cite[Section 2.3.2]{schmitt08}.
\end{rem}

The representation $\sigma$ induces a weighted flag $(E_{\sigma,\bullet},\tup{\alpha}_\sigma)$ of $E_\sigma$. Its restriction to $x_0$ yields a flag of $E_{\sigma|\{x_0\}}$ and we set
\[                                                                                                                                                         
  \mu_2(E_\bullet,\tup{\alpha},s):=\mu(E_{\sigma,\bullet|\{x_0\}},\tup{\alpha}_\sigma,[s])\,.                                                                                                                                                       
\]

\begin{mydef} \label{def:stability}
Let $\delta_1,\delta_2\in\QQ_{>0}$. We call a decorated swamp $(E,L,\phi,s)$ \emph{$(\delta_1,\delta_2)$-\textup{(}semi\nobreakdash-\textup{)}stable} if the relation
 \begin{equation}
  M(E_\bullet,\tup{\alpha}) + \delta_1 \mu_1(E_\bullet,\tup{\alpha},\phi)
+ \delta_2\mu_2(E_\bullet,\tup{\alpha},s) (\ge) 0
 \end{equation}
 holds for any weighted filtration $(E_\bullet,\tup{\alpha})$ of $E$.
\end{mydef}

\begin{rem} \label{rem:mu_1,2}
 Since we assume $\rho$ and $\sigma$ to be homogeneous, there are numbers $a_m,b_m,c_m$, such that the $\GL(r)$-module $V_m$ is a direct summand of $(\CC^r)_{a_m,b_m,c_m}$, $m=1,2,$ (see \eqref{eq:V_abc}).
 We thus have surjections $p_1:E_{a_1,b_1,c_1}\to E_\rho$ and $p_2:E_{a_2,b_2,c_2}\to E_{\sigma}$.
 
 For $m=1,2$ and a tuple $\tup{i}\in
I_m:=\{1\,\ldots,k\}^{a_m}$ we define
 \[
  E^{\otimes\tup{i}}:=(E_{i_1}\otimes \cdots \otimes
E_{i_{a_m}})^{\oplus b_m}\otimes \left(\bigwedge^r
E\right)^{\otimes -c_m} \subset E_{a_m,b_m,c_m}\,.
 \]
 Using \eqref{eq:stability_V_abc} we easily find
 \begin{align} \label{eq:tmp_stab1}
  \mu_1(E_\bullet,\tup{\alpha},\phi)&=-\min\left\{\sum_{j=1}^l
\alpha_j(a_1\rk(E_j)-r \nu_j(\tup{i}))
\Bigm|\,(\phi\circ p_1)_{|E^{\otimes\tup{i}}} \neq 0 \,,\tup{i}\in I_1 \right\},\\
  \label{eq:tmp_stab2}
  \mu_2(E_\bullet,\tup{\alpha},s)&=-\min\left\{ \sum_{j=1}^l
\alpha_j(a_1\rk(E_j)-r \nu_j(\tup{i}))  \,\Bigm|\,
(s\circ p_2)_{|E_{|\{x_0\}}^{\otimes\tup{i}}}\neq 0\,, \tup{i}\in I_2\right\},
 \end{align}
with $r=\rk(E)$ and $\nu_j(\tup{i}):=\#\{i\in \tup{i}\,|\,i\le j \}$.
\end{rem}

\subsection{S-Equivalence}
\begin{mydef}
Let $(E,L,\phi,s)$ be a $(\delta_1,\delta_2)$-semistable decorated swamp. We call
 a weighted flag $(E_\bullet,\tup{\alpha})$ \emph{critical} if
\begin{equation*}
  \label{eq:swamp_M=0}
 M(E_\bullet,\tup{\alpha})+\delta_1\mu_1(E_\bullet,\tup{\alpha},
\phi)+\delta_2\mu_2(E_\bullet,\tup{\alpha})=0\,.
\end{equation*}
\end{mydef}
Let $\la$ be the one-parameter subgroup of $\GL(r,\CC)$ defined by the standard flag of $\CC^r$ and the weight vector $\tup{\alpha}$. The datum of $E_\bullet$ corresponds to a reduction of the structure group $\beta:X\to P(E)/Q(\la)$. Let $P'$ be the $Q(\la)$-bundle $\beta^*P(E)$. Let $\pi:Q(\la)\to L(\la)$ be the projection to the Levi factor. We define $E_{\gr}$ to be the associated vector bundle $\pi_*P'\times^{\GL(r)}\CC^r$. Note that 
$E_{\gr}= \bigoplus_{i=0}^k E_{i+1}/E_i$.

The one-parameter subgroup $\rho_*\la$ of $\GL(V_1)$ defines a weighted flag $(V^{(1)},\tup{\alpha}^{(1)})$ of $V_1$. Since $\rho(Q(\la))\subset Q(\rho_*\la)$ the reduction $\beta$ also defines a flag $E^{(1)}_\bullet$ of $E_\rho\cong \rho_*P'\times^{\GL(V_1)}V_1$. The image of the Levi factor $L(\la)$ is contained in $L(\rho_*\la)$ of $Q(\rho_*\la)$. Let $\pi_\rho:Q(\rho_*\la)\to L(\rho_*\la)$ be the projection. Since $\rho\circ \pi=\pi_\rho\circ\rho$, we find that $\rho_*\pi_*P'\cong \pi_{\rho*}\rho_*P'$. Hence we get
\[
 E_{\gr,\rho}= \bigoplus_{i=1}^{k+1} E^{(1)}_i/E^{(1)}_{i-1}\,.
\]
Similarly, the flag $E_\bullet$ defines weights $\tup{\alpha}^{(2)}$ and a flag $E^{(2)}_\bullet$ of $E_{\sigma}$ such that
\[
 E_{\gr,\sigma}= \bigoplus_{i=1}^{k+1} E^{(2)}_i/E^{(2)}_{i-1}\,.
\]

Set $i_0:=\min\{i\in \{1,\ldots,k+1\}\,|\,\phi_{|E^{(1)}_i}\neq 0\}$ and $j_0:=\min\{j\in \{1,\ldots,k+1\}\,|\,s_{|E^{(2)}_j}\neq 0\}$. Then the restrictions $\phi_{|E^{(1)}_{i_0}}$ and $s_{|E^{(2)}_{j_0}}$ define non-trivial homomorphisms
\begin{align*}
 \phi_0:E^{(1)}_{i_0}/E^{(1)}_{i_0-1}\to L \,, \qquad s_0:E^{(2)}_{j_0}/E^{(2)}_{j_0-1}\to \CC\,.
\end{align*}
Let $\phi_\gr:E_{\gr,\rho}\to L$ and $s_\gr:E_{\gr,\sigma}\to \CC$ be the compositions of the natural projections with these homomorphisms. We call
\[
 \opname{df}_{(E_\bullet,\tup{\alpha})}(E,L,\phi,s):=(E_{\gr},L,\phi_{\gr},s_{\gr})\,.
\]
the \emph{admissible deformation} of $(E,L,\phi,s)$ along $(E_\bullet,\tup{\alpha})$.

\begin{mydef} \label{def:S-equivalence}
We define \emph{S-equivalence} as the equivalence relation generated by isomorphisms and the relations
\[
 (E,L,\phi,s)\sim_S \opname{df}_{(E_\bullet,\tup{\alpha})}(E,L,\phi,s)
\]
for all critical flags $(E_\bullet,\tup{\alpha})$.
\end{mydef}

\subsection{Parameterized Families}
In order to define a moduli functor we need to introduce the notion of a parameterized family: Denote by $\Jac^l(X)$ the Jacobian of line bundles of degree $l$ on $X$ and fix a Poincar\'e bundle $\mc{L}$ on $\Jac^l\times X$.
\begin{mydef}
 A \emph{family of decorated swamps of type $(d,l)$} parameterized by a scheme $S$ is a tuple $D_S=(E_S,\ka_S,N_{1,S},N_{2,S},\phi_S,s_S)$ consisting of
 \begin{itemize}
  \item  a vector bundle $E_S$ of rank $r$ on $S\times X$, such that for any point $s\in S$ the restriction of $E_S$ to $\{s\}\times X$ is of degree $d$,
  \item a morphism $\kappa_S:S\to \Jac^l$,
  \item a line bundle $N_{1,S}$ on $S$,
  \item a homomorphism $\phi_S:E_{S,\rho}\to \pr_S^*N_{S}\otimes(\ka_S\times\id_X)^*\mc{L}$ such that for any $s\in S$ the restriction to $\{s\}\times X$ is non-trivial
  \item and a surjective homomorphism $s_S:E_{S,\sigma|S\times\{x_0\}}\to N_{2,S}$.
 \end{itemize}

Two families $D_S$ and $D'_S$ are considered \emph{isomorphic}, if $\ka_S=\ka'_S$ holds and there are
\begin{itemize}
 \item a line bundle $L_S$ on $S$ and an isomorphism $f:E_S\to E'_S\otimes\pr_S^*L_S$,
 \item an isomorphism $h_1:N_{1,S}\to N'_{1,S}\otimes L_S^{\otimes \deg(\rho)}$ with 
 \[
 (\phi'_S\otimes \id_{L_S^{\otimes \deg(\rho)}})\circ f_\rho=(\pr_S^*h_1 \otimes\id_{(\ka_S\otimes\id_X)^*\mc{L}})\circ\phi_S\,,
 \]
 \item and an isomorphism $h_2:N_{2,S}\to N'_{2,S}\otimes L^{\deg(\sigma)}_S$ such that
  \[
(s'_S\otimes \id_{L_S^{\otimes \deg(\sigma)}})\circ f_{\sigma|S\times\{x_0\}}=h_2\circ s_S\,.
\]
\end{itemize}
\end{mydef}

Note that for a point $s\in S$ a family $D_S$ defines a decorated swamp
\[
D_{S|s}:=(E_{S|\{s\}\times X},(\ka(s)\times\id_X)^*\mc{L},\phi_{S|\{s\}\times X},s_{S|\{s\}}).
\]
\begin{mydef}
A family $D_S$ parameterized by $S$ is called \emph{$(\delta_1,\delta_2)$-\textup{(}semi\nobreakdash-\textup{)}stable} if for any $s\in S$ the restriction $D_{S|s}$ is $(\delta_1,\delta_2)$-(semi\nobreakdash-)stable.
\end{mydef}

We are now ready to state the main result of this work. Recall that in \autoref{rem:mu_1,2} we have fixed numbers $a_2,b_2,c_2$ such that the representation $\sigma$ is a direct summand of the natural representation on $(\CC^r)_{a_2,b_2,c_2}$.

\begin{thm} \label{main_thm}
 For $a_2\delta_2<1$ the \textup{(}projective\textup{)} coarse moduli space  
 of $(\delta_1,\delta_2)$-\textup{(}semi\nobreakdash-\textup{)}stable decorated swamps of type $(d,l)$ exists.
\end{thm}

The rest of this article is devoted to the construction of this moduli space. The proof of the theorem is given in \autoref{sec:proof_of_main_thm}.

%%%%%%%%%%%%%%%%%%%%%%%%%%%%%%%%%%%%%%%
\section{The Parameter Space}
As a first step towards the construction of the moduli space we construct a parameter space with a locally universal family of decorated swamps. In order to be able to apply GIT we then construct an equivariant morphism from this parameter space into a projective space.

\subsection{Construction of the Parameter Space}
We start with a result on boundedness.
\begin{prop} \label{prop:swamps_bounded}
 There is a constant $C$ such that for any $(\delta_1,\delta_2)$-semistable decorated swamp
$(E,L,\phi,s)$ any non-trivial proper subbundle $F\subset E$ satisfies
 \[
  \mu(F)< \mu(E)+C\,.
\]
\end{prop}
\begin{proof}
 Consider the weighted flag $(0\subset F\subset E,(1))$ of $E$. Semistability implies
 \[
\deg(E)\rk(F)-\deg(F)\rk(E)+\delta_1\mu_1(E_\bullet,\tup{\alpha},\phi)+\delta_1\mu_2(E_\bullet,\tup{\alpha},s) \ge
0\,.
 \]
 Using \eqref{eq:tmp_stab1} and \eqref{eq:tmp_stab2} we find
 \[
 \mu(F)\le \mu(E)+(\delta_1a_1 +
\delta_2a_2)\frac{r-1}{r}\,.
\]
\end{proof}
By standard arguments it follows that there is a natural number $n_0$ such that for $n\ge n_0$ and a $(\delta_1,\delta_2)$-semistable decorated swamp $(E,L,\phi,s)$ we have $h^1(E(n))=0$ and $E(n)$ is globally generated. We fix $n\ge n_0$ and set $p(n):=d+r(n+1-g)$ and $Y:=\CC^{p(n)}$. 
\begin{mydef}
A \emph{family of decorated quotient swamps} of type $(d,l)$ parameterized by a scheme $S$ is a tuple $(q_S,\ka_S,N_{1,S},N_{2,S},\phi_S,s_S)$ where $q_S:\pr_X^*\mc{O}_X(-n) \otimes Y\to E_S$ is a vector bundle quotient on $S\times X$ such that $(E_S,\ka_S,N_{1,S},N_{2,S},\phi_S,s_S)$ is a family of decorated swamps of type $(d,l)$ and  
\[
 \pr_{S *}(q_s\otimes\id_{\pr_X^*\mc{O}_X(n)}):Y\otimes\mc{O}_S\to \pr_{S *}(E_S\otimes\pr_X^*\mc{O}_X(n))
\]
 is an isomorphism.
\end{mydef}

These families define the moduli functor of decorated quotient swamps.
\begin{prop} \label{prop:parameter_space}
 The fine moduli space of decorated quotient swamps of type $(d,l)$ exists.
\end{prop}
\begin{proof}
We construct the moduli space over the parameter space of swamps from \cite[\S 2.3.5]{schmitt08}: let $\Quot_n$ be Grothendieck's Quot scheme for quotients of rank $r$ and degree $d$ of $Y\otimes\mc{O}_X(-n)$ on $X$ and let $q:Y\otimes\pr_X^*\mc{O}_X(-n)\to Q$ be the universal quotient on $\Quot_n\times X$. There is an open subscheme $\Quot_n^0$ consisting of the points $s\in \Quot_n$ such that $Q_{s}$ is a vector bundle and $H^0(q_s(n))$ is an isomorphism. We set $E:=Q_{|\Quot_n^0\times X}$. For sufficiently large $m$ the pushforward of the sheaves
\begin{align*}
  \mc{F}:&=(\pr_{\Quot^0_n}\times \id_X)^*E_\rho \otimes \pr_X^*\mc{O}_X(m)\,,\\
  \mc{G}:&=(\pr_{\Jac^l}\times\id_X)^*\mc{L} \otimes \pr_X^*\mc{O}_X(m)\,.
\end{align*}
to $P_0:=\Quot^0_n\times\Jac^l$ will be locally free. On the projective bundle
\[
 P_1:=\PP(\mc{H}om(\pr_{P_0 *}\mc{F},\pr_{P_0 *}\mc{G})^\vee)
\]
 we have the tautological homomorphism
\[
 f:\pr_{P_0}^*\pr_{P_0 *}\mc{F} \to \pr_{P_0}^*\pr_{P_0 *}\mc{G}\otimes \mc{O}_{P_1}(1)\,.
\]
Let $K$ be the kernel of the surjective evaluation homomorphism 
\[
\ev_{\mc{F}}:\pr_{P_0}^*\pr_{P_0 *}\mc{F} \to (\pr_{P_0}\times \id_X)^*\mc{F}
\]
and denote by $h$ the restriction of $(\ev_\mc{G}\otimes\id_{\mc{O}_{P_1}(1)})\pr_{P_1}^*f$ to $K$. Then, there is a closed subscheme $I:=V(h)\subset P_1$ such that a morphism $\psi:T\to P_1$ factors via $V(h)$ if and only if $\psi^*h$ is trivial (see \cite[Proposition 2.3.5.1]{schmitt08}). Thus, on $I\times X$ we have a universal homomorphism
\[
 \phi: (\pr_{\Quot_n^0}\times\id_X)^*E_\rho \to (\pr_{\Jac^l}\times\id_X)^*\mc{L} \otimes \mc{O}_{P_1}(1)\,.
\]

Over $I$ we consider the bundle
\[
P:=\PP\left((\pr_{\Quot_n^0}\times\id_X)^*E_\sigma\right)_{|I\times\{x_0\}}\,.
\]
On $P\times X$ we have the quotient 
\[
\tilde{q}:Y\otimes\pr_X^*\mc{O}_X(-n)\to \tilde{E}:=\pr_{\Quot_n^0}^*E\,,
\]
the homomorphism
\[
 \tilde{\phi}:\tilde{E}_\rho \to \tilde{L}\otimes \pr_{P}^*\tilde{N}_1
\]
with $\tilde{\ka}:=\pr_{\Jac^l}$, $\tilde{L}:=(\tilde{\ka}\times \id_X)^*\mc{L}$ and
$\tilde{N}_1:=\pr_{P_1}^* \mc{O}_{P_1}(1)$, and the tautological homomorphism
\[
 \tilde{s}:\tilde{E}_{\sigma|P\times\{x_0\}}\to \tilde{N}_2:=\mc{O}_{\PP(P)}(1)\,.
\]
It is a routine exercise to check that the family $(\tilde{q},\tilde{\ka},\tilde{N}_1,\tilde{N}_2,\tilde{\phi},\tilde{s})$ is universal.
\end{proof}

%%%%%%%%%%%%%%%%%%%%%%%%%%%%%
\subsection{The Gieseker Space}
We recall the construction of the Gieseker space from \cite[\S 2.3.5]{schmitt08}. Let $\Jac^d$ be the Jacobian of line bundles of degree $d$ on $X$ and choose a Poincar\'e bundle $\mc{P}$ on $\Jac^d\times X$. For sufficiently large $n$ the sheaf
\[
 \mc{G}_1:=\mc{H}om\left( \bigwedge^r Y \otimes\mc{O}_{\Jac^d},\pr_{\Jac^d *}(\mc{P}\otimes \pr_X^*\mc{O}_X(rn) )\right)
\]
on $\Jac^d$ is locally free. We define $\Gies_1:=\PP(\mc{G}_1^\vee)$. Without loss of generality we may assume $\mc{O}_{\Gies_1}(1)$ to be very ample. 

Let $p:\Quot_n^0\to \Jac^d$ denote the morphism determined by $\det(E)$. Then, there is a line bundle $\mc{A}$ on $\Quot_n^0$ such that $\det(E)\cong \pr_{\Quot_n^0}^*\mc{A}\otimes (p\times\id_X)^*\mc{P}$. The pushforward of the homomorphism
\[
 \bigwedge^r\left(q\otimes\id_{\pr_X^*\mc{O}_x(n)} \right):\bigwedge^r Y\otimes \mc{O}_{\Quot_n^0}\times X\to \det(E)\otimes \pr_X^*\mc{O}_X(nr)
\]
from $\Quot_n^0\times X$ to $\Quot_n^0$ determines a $\GL(Y)$-equivariant morphism $\iota_1:\Quot_n^0\to \Gies_1$ over $\Jac^d$ with $\iota_1^*\mc{O}_{\Gies_1}(1)=\mc{A}$.

On $J:=\Jac^d\times \Jac^l$ consider the locally free sheaf
\[
 \mc{G}_2:=\mc{H}om\left(Y_{a_1,b_1}\otimes\mc{O}_{J},\pr_{J *}\left(\pr_{\Jac^d\times X}^*\mc{P}^{\otimes c_1}\otimes \pr_{\Jac^l\times X}^*\mc{L}\otimes\pr_X^*\mc{O}_X(na_1)\right)\right)\,.
\]
and set $\Gies_2:=\PP(\mc{G}_2)$. Again one may assume $\mc{O}_{\Gies_2}(1)$ very ample.

On $I\times X$ we have the surjection
\[
 Y_{a_1,b_1}\otimes \pr_X^*\mc{O}_X(-a_1n)\to E_{a_1,b_1}\to E_\rho\otimes \det(E)^{\otimes c_1}\,.
\]
The composition with $\tilde{\phi}\otimes\id_{\det(E)^{\otimes c_1}}$ leads to a homomorphism
\[
 Y_{a_1,b_1}\otimes\mc{O}_{I\times X} \to L\otimes \pr_I ^*N_1\otimes \det(E)^{\otimes c_1}\otimes \pr_X^*\mc{O}_X(na_1)\,.
\]
The pushforward to $I$ yields a homomorphism
\[
 Y_{a_1,b_1}\otimes \mc{O}_I \to \pr_{I *}(L\otimes \pr_{\Jac^d \times X}\mc{P}^{\otimes c_1} \otimes \pr_X^*\mc{O}_X(na_1) )\otimes \mc{B}
\]
with $\mc{B}:=N_1\otimes \mc{A}$. This defines a $\GL(Y)$-equivariant morphism $\iota_2:I \to \Gies_2$ over $\Jac^d$ with $\iota_2^*\mc{O}_{\Gies_2}(1)=\mc{B}$. The Gieseker morphism from \cite[\S 2.3.5]{schmitt08} is the injective equivariant morphism
\[
  (\iota_1\circ \pr_{\Quot_n^0},\iota_2): I \to \Gies_1\times_{\Jac^d}\Gies_2\,.
\]

On $P\times X$ we have the surjection
\[
 Y_{a_2,b_2}\otimes\mc{O}_{P\times X}(-a_2n) \to E_{a_2,b_2}\to E_\sigma\otimes \det(E)^{\otimes c_2}\,.
\]
Restricting to $P\times\{x_0\}$ and composing with $\tilde{s}$ leads to a homomorphism
\[
 Y_{a_2,b_2}\otimes \mc{O}_P \to  N_2\otimes \pr_{\Quot_n^0}^*\mc{A}^{\otimes c_2} \otimes (\pr_{\Jac^d\times X}^*\mc{P}^{\otimes c_2})_{|P}\,.
\]
This defines a $\GL(Y)$-equivariant morphism
\[
 \iota_3:P \to \PP(Y_{a_2,b_2})=:\Gies_3\,.
\]
\begin{mydef} \label{def:Giesker_morphism}
 We define the \emph{Gieseker morphism} as
\[
 \gies_n:= (\iota_1\circ \pr_{\Quot_n^0},\iota_2,\iota_3):P \to \Gies_n:=\Gies_1\times_{\Jac^d}\Gies_2\times\Gies_3 \,.
\]
\end{mydef}
It is easy to see that $\gies_n$ is injective and $\GL(Y)$-equivariant.

\subsection{GIT Stability in the Gieseker Space}
 Let $L_d$ and $L$ be line bundles of degree $d$ and $l$ respectively. The fiber of $\Gies_n$ over the corresponding point in $\Jac^d\times \Jac^l$ is isomorphic to
 \[
  \PP\left(\Hom\left(\bigwedge^r Y,H^0(L_d(nr))\right)^\vee\right)\times \PP\left(Y_{a_1,b_1},H^0(L_d^{\otimes c_1}\otimes L(a_1n)^\vee\right) \times \PP(Y_{a_2,b_2})\,.
 \]
Let $q:Y\otimes \mc{O}_X(-n)\to E$ be a generically surjective morphism of vector bundles with $\rk(E)=r$ and $\det(E)=L_d$, and let $[M]\in \Gies_1$ denote the point defined by $q$. For a one-parameter subgroup $\la$ of $\SL(Y)$ with associated weighted flag $(Y_\bullet,\tup{\alpha})$ we calculate
\begin{equation} \label{eq:stab_in_Gies_0}
 \mu(\la,[M])=\sum_{j=1}^{\len(Y_\bullet)}\alpha_j\left(p(n)\rk(F_j)-\rk(E)\dim(Y_j) \right)\,.
\end{equation}
Here $F_j$ denotes the subsheaf $q(Y_j\otimes\mc{O}_X(-n))\subset E$ generated by $Y_j$.

For two points $T_1\in  \PP\left(Y_{a_1,b_1},H^0(L_d^{\otimes c_1}\otimes L(a_1n)^\vee\right) $ and $T_2\in \PP(Y_{a_2,b_2})$ we find
\begin{align} \label{eq:stab_in_Gies_1}
 \mu(\lambda,[T_1])&=-\min\left\{\sum_{j=1}^{\len(Y_\bullet)}
\alpha_j(a_1\dim(Y_j)-p(n)\nu_j(\tup{i}))\,\biggm|\,\tup{i}\in
I_1:T_{1|Y^{\otimes\tup{i}}}\neq 0 
\right\}\,, \\
\label{eq:stab_in_Gies_2}
 \mu(\lambda,[T_2])&=-\min\left\{\sum_{j=1}^{\len(Y_\bullet)}
\alpha_j(a_2\dim(Y_j)-p(n)\nu_j(\tup{i}))\,\biggm|\,\tup{i}\in
I_2:T_{2|Y^{\otimes\tup{i}}}\neq 0 
\right\}\,.
\end{align}
Here, $I_1$, $I_2$ and $\nu_j$ are as in equations \eqref{eq:tmp_stab1} and \eqref{eq:tmp_stab2}. For $\tup{i}\in I_m$ we define
\[
 Y^{\otimes \tup{i}}:= \bigotimes_{j=1}^{a_m} Y_{i_j} \subset Y_{a_m,b_m}\,,\qquad m=1,2\,.
\]
In particular we have the estimate
\begin{equation} \label{eq:estimate_mu(T_2)}
 -a_2 \sum_{j=1}^{\len(Y_\bullet)}\alpha_j\dim(Y_j)  \le \mu(\lambda,[T_2])\le a_2\sum_{j=1}^{\len(Y_\bullet)}\alpha_j \big(p(n)-\dim(Y_j)\big)\,.
\end{equation}

We linearize the action of $\GL(Y)$ in the line bundle
\[
\mc{O}_{\Gies_n}(\eta,\theta_1,\theta_2):=
\mc{O}_{\Gies_{1}}(\eta)\boxtimes \mc{O}_{\Gies_2}(\theta_1)\boxtimes
\mc{O}_{\Gies_3}(\theta_2) 
\]
with 
\begin{equation} \label{eq:Linearisierung}
 \eta:=z (p(n)-a_1\delta_1-a_2\delta_2)   \,, \qquad
\theta_1:=zr\delta_1\,, \qquad \theta_2:=zr\delta_2\,.
\end{equation}
Here, $z$ is a natural number such that $\eta$, $\theta_1$ and $\theta_2$ are positive integers.

%%%%%%%%%%%%%%%%%%%%%%%%%%%%%%%%%%%%%%%%%%%%%%%%%%%%%%%%%%%%%
\section{Comparison of Stability Conditions}
The aim of this section is to show that for large enough $n$ GIT stability of the image under the Gieseker morphism of a point in the parameter space is equivalent to stability of the corresponding decorated swamp. For this purpose we introduce a third notion of stability.

\subsection{Section Stability}
For a weighted flag $(E_\bullet,\tup{\alpha})$ and a number $n$ we define the quantity
\[
 M^{\textnormal{s}}(E_\bullet,\tup{\alpha},n):=\sum_{i=1}^k \alpha_i\left( h^0(E(n))\rk(E_i)-h^0(E_i(n))\rk(E)\right)\,.
\]
\begin{mydef}
 We call a decorated swamp $(E,L,\phi,s)$ \emph{$(\delta_1,\delta_2,n)$-section-\textup{(}semi\nobreakdash-\textup{)}stable} if any weighted filtration $(E_\bullet,\tup{\alpha})$ of $E$ satisfies
\[
 M^{\textnormal{s}}(E_\bullet,\tup{\alpha},n)+\delta_1\mu_1(E_\bullet,\tup{\alpha},\phi)+\delta_2\mu_2(E_\bullet,\tup{\alpha},s) (\ge) 0\,.
\]
\end{mydef}
\begin{lemma}
 There is an $n_2\ge n_1$ such that for $n\ge n_2$ and a $(\delta_1,\delta_2,n)$-section-semistable swamp $(E,L,\phi,s)$ we have $h^1(E(n))=0$.
\end{lemma}
\begin{proof}
 Assume $h^1(E(n))=h^0(E^\vee(-n)\otimes \omega_X)\neq 0$ and let $f:E\to \omega_X(-n)$ be a non-trivial homomorphism. Section semistability for the weighted flag $(0\subset \ker(f)\subset E,(1))$ implies
 \[
  h^0(F(n))\rk(E)\le h^0(E(n))\rk(F)+(a_1\delta_1+a_2\delta_2)\rk(F)\,.
 \]
Using $h^0(E(n))-g\le h^0(F(n))$, $\rk(F)=\rk(E)-1$ and the Riemann--Roch theorem we find
\[
 \mu(E)+n+1-g \le \frac{h^0(E(n))}{\rk(E)} \le g+\frac{\rk(E)-1}{\rk(E)}(a_1\delta_1+a_2\delta_2)
\]
Thus, we get a contradiction for $n\ge 2g-\mu(E)+(a_1\delta_1+a_2\delta_2)$.
\end{proof}
The above lemma together with the Riemann--Roch theorem implies:
\begin{cor}
 For $n\ge n_2$ and a $(\delta_1,\delta_2,n)$-section-semistable swamp $(E,L,\phi,s)$ we have 
 \[
  M(E_\bullet,\tup{\alpha}) \ge M^{\textnormal{s}}(E_\bullet,\tup{\alpha},n)\,.
 \]
\end{cor}

\begin{lemma}
 \label{lem:h1=0=>positive}
 There is an $n_3\ge n_2$ such that for $n\ge n_3$ and a $(\delta_1,\delta_2)$-semi-stable swamp $(E,L,\phi,s)$ any subbundle $F\subset E$ with $h^1(F(n))\neq 0$ satisfies
 \[
  h^0(E(n))\rk(F)-h^0(F(n))\rk(E)-\rk(F)(a_1\delta_1+a_2\delta_2) > 0\,.
 \]
\end{lemma}
\begin{proof}
 By \autoref{prop:swamps_bounded} there is a constant $C$ such that $\mu(F)\le \mu(E)+C$ holds for all $(\delta_1,\delta_2)$-semi-stable decorated swamps $(E,L,\phi,s)$ and all subbundles $F$ of $E$. We set
 \[
  C':= g+ C+\frac{a_1\delta_1+a_2\delta_2}{\rk(E)}
 \]
and divide the set of isomorphism classes of subbundles of a vector bundle $E$ occurring in some semi-stable decorated swamp into two subsets
\begin{align*}
 A:&=\{[F]\,|\, \mu_{\min} < \mu(E)+C -\rk(F)C'\}\,, \\
 B:&=\{[F]\,|\, \mu_{\min} \ge \mu(E)+C -\rk(F)C'\}\,.
\end{align*}

Suppose $[F]\in A$. The Le Potier--Simpson estimate \cite[Lemma 7.1.2]{LePot97} yields
\[
 h^0(F(n))\le (\rk(F)-1)[\mu_{\max}(F(n))+1]_+ +[\mu_{\min}(F(n))+1]_+\,.
\]
The inequality defining $A$ and the Riemann--Roch theorem then imply
\[
 h^0(F(n))< \rk(F) \frac{h^0(E(n))}{\rk(E)}-\frac{\rk(F)}{\rk(E)}(a_1\delta_1+a_2\delta_2)\,.
\]

For $[F]\in B$, the definition of $B$ and \autoref{prop:swamps_bounded} lead to
\[
  \mu(E)+C-\rk(F)C'  \le \mu(F) < \mu(E) + C\,.
\]
Thus, the set $\{\deg(F)\,|\,[F]\in B\}$ is finite and hence $B$ is bounded. In particular, there is an $n_3\ge n_2$ such that $h^1(F(n))=0$ for $n\ge n_3$.
\end{proof}

\begin{mydef} \label{def:decomposition}
 Let $(E_\bullet,\tup{\alpha})$ be a weighted flag. We decompose the set of indices into two subsets $I^A=\{i^A_1< \ldots<i^A_{k^A}\}$ and $I^B=\{i^B_1< \ldots<i^B_{k^B}\}$ such that $i\in I^A$ if $h^1(E_i(n))\neq 0$ and $i\in I^B$ otherwise. Then we define the flags
 \[
  E^{A/B}\,: \, 0 \subset E_{i^{A/B}_1} \subset \cdots \subset E_{i^{A/B}_{k^{A/B}}} \subset E
 \]
and the weight vectors $\tup{\alpha}^{A/B}:=(\alpha_{i^{A/B}_1},\ldots,\alpha_{i^{A/B}_{k^{A/B}}})$.
\end{mydef}

\begin{prop}
 For $n\ge n_3$, a $(\delta_1,\delta_2)$-semi-stable swamp $(E,L,\phi,s)$ and a weighted flag $(E_\bullet,\tup{\alpha})$ 
 \begin{align*}
   &M^{\textnormal{s}}(E_\bullet,\tup{\alpha},n)+\delta_1\mu_1(E_\bullet,\tup{\alpha},\phi)+\delta_2\mu_2(E_\bullet,\tup{\alpha},s)\\
   \ge &\,M(E^B_\bullet,\tup{\alpha}^B)+\delta_1\mu_1(E^B_\bullet,\tup{\alpha}^B,\phi)+\delta_2\mu_2(E^B_\bullet,\tup{\alpha}^B,s)\,.
 \end{align*}
\end{prop}
\begin{proof}
 i) Suppose $(E_\bullet,\tup{\alpha})=(E^B_\bullet,\tup{\alpha}^B)$, i.e., $h^1(E_i(n))=0$ for all $i$. Then, the Riemann--Roch theorem implies
 \[
  M^{\textnormal{s}}(E_\bullet,\tup{\alpha},n)= M(E_\bullet,\tup{\alpha})\,.
 \]
 
 ii) If there is an index $i$ with $h^1(E_i(n))\neq 0$ then we use the decomposition from \autoref{def:decomposition}. Using the estimate \eqref{eq:estimate_V_abc} we get
  \begin{align*}
 &
M^{\textnormal{s}}(E_\bullet,\tup{\alpha},n)+\delta_1\mu_1(E_\bullet,\tup{\alpha},
\phi)+\delta_2\mu_2(E_\bullet,\tup{\alpha},s)\\
 \ge&
M^{\textnormal{s}}(E_\bullet^B,\tup{\alpha}^B,n)+\delta_1\mu_1(E_\bullet^B,\tup{
\alpha}^B,\phi)+\delta_2\mu_2(E_\bullet^B,\tup{\alpha}^B,s)\\
  & +
M^{\textnormal{s}}(E_\bullet^A,\tup{\alpha}^A,n)-(\delta_1a_1+\delta_2a_2)\sum_{j=1
}^{\len(E_\bullet^A)}\alpha^A_j \rk(E^A_j)\,.
 \end{align*}
The last term is positive by \autoref{lem:h1=0=>positive}. The claim now follows from part (i).
\end{proof}

\begin{cor} \label{cor:stable<=>section_stable}
 For $n\ge n_3$ a decorated swamp $(E,L,\phi,s)$ is $(\delta_1,\delta_2)$-(semi-)stable if and only if it is $(\delta_1,\delta_2,n)$-section-(semi-)stable. In this case every weighted flag $(E_\bullet,\tup{\alpha})$ satisfies
\begin{align*}
 & M^{\textnormal{s}}(E_\bullet,\tup{\alpha},n)+\delta_1\mu_1(E_\bullet,\tup{\alpha},
\phi)+\delta_2\mu_2(E_\bullet,\tup{\alpha},s)(\ge) 0\\
\Longleftrightarrow \quad & M(E_\bullet,\tup{\alpha})+\delta_1\mu_1(E_\bullet,\tup{\alpha},
\phi)+\delta_2\mu_2(E_\bullet,\tup{\alpha},s)(\ge) 0\,.
 \end{align*}
\end{cor}

\subsection{Slope and GIT Stability}
In the following let $p\in P$ be a point and let $(E,L,\phi,s)$ be the decorated swamp defined by $p$.

If $(E_\bullet,\tup{\alpha})$ is a weighted flag of $E$, we associate with it a weighted flag $\Gamma_p(E_\bullet,\tup{\alpha})$ in the following way: Set $U_j:=H^0(q_p(n))^{-1}H^0(E_j(n))$ and let $Y_\bullet$ be the induced flag of $Y$. For $1\le h\le \len(Y_\bullet)$ we define $J(h):=\{j\,|\, U_j=Y_h\}$ and
\[
 \beta_h:=\sum_{j\in J(h)} \alpha_j \,, \qquad 1\le h\le \len(Y_\bullet)\,.
\]
Note that we also allow the trivial flag $0\subset Y$ with weight $\tup{\beta}\in\{0\}$. Then we set $\Gamma_p(E_\bullet,\tup{\alpha}):=(Y_\bullet,\tup{\beta})$.

\begin{prop}\phantomsection \label{prop:GIT-stable=>stable}
 \begin{enumerate}
\item For a weighted flag $(E_\bullet,\tup{\alpha})$ and a one-parameter subgroup $\la$ of $\SL(Y)$ with associated weighted flag $\Gamma(E_\bullet,\tup{\alpha})=(Y_\bullet,\tup{\beta})$ we have
\[
  \frac{\mu(\la,\gies_n(t))}{zp(n)} \le M^{\textnormal{s}}(E_\bullet,\tup{\alpha},n)+\delta_1\mu_1(E_\bullet,\tup{
\alpha},\phi)+\delta_2\mu_2(E_\bullet,\tup{\alpha},s)\,.
 \]

\item If equality holds, one has $\len(E_\bullet)=\len(Y_\bullet)$ and $E_j$ is generically generated by $Y_j$ for $1\le j \le \len(E_\bullet)$.
\end{enumerate}
\end{prop}
\begin{proof}
 Let $F_h\subset E$ be the subbundle generated by $Y_h\otimes\mc{O}_X(-n)$. We define $j(h):=\min J(h)$. For $1\le h \le \len(Y_\bullet)$ and $j\in J(h)$ we have $\dim(Y_h)=h^0(E_j(n))$ and $\rk(F_h)\le \rk(E_j)$. For $j\in J(0)$ we find $h^0(E_j(n))=0$. Now \eqref{eq:stab_in_Gies_0} implies
  \begin{align*}
 \mu(\lambda,[M]) \le& \sum_{j=1}^{\len(E_\bullet)} \alpha_j\left(p(n)\rk(E)-h^0(E_j(n))\rk(E)\right)\,,
\end{align*}
Note that equality can only hold, if $J(0)=\varnothing$ and $\rk(F_h)=\rk(E_j)$ for $j\in J(h)$.

Let $\tup{h}^{(1)}$ and $\tup{h}^{(1)}$ be tuples for which the minimum in \eqref{eq:stab_in_Gies_1} and \eqref{eq:stab_in_Gies_2} is attained. For an index $0\le j \le \len(E_\bullet)$ let $1\le h(j)\le \len(Y_\bullet)$ be the index with $Y_{h(j)}=U_j$. Since $\nu_{h(j)}(\tup{h}^{(m)})=0$ and $h^0(E_j(n))=0$ for $m=1,2$, $j\in J(0)$ we find
\begin{align*}
 \mu(\lambda,[T_1]) &= \sum_{j=1}^{\len(E_\bullet)}
\alpha_{j}\left(p(n)\nu_{h(j)}(\tup{h}^{(1)})-a_1h^0(E_j(n))\right)\,,\\
 \mu(\lambda,[T_2]) &= \sum_{j=1}^{\len(E_\bullet)}
\alpha_{j}\left(p(n)\nu_{h(j)}(\tup{h}^{(2)})-a_2h^0(E_j(n))\right)\,.
\end{align*}

With our choice of a linearization the above calculations show
\begin{align*}
 \mu(\la,\gies_n(p))=&  \eta \mu(\lambda,[M])+\theta_1
\mu(\lambda,[T_1])+\theta_2\mu(\lambda,[T_2])\\
  \le & zp(n)\left[\sum_{j=1}^{\len(E_\bullet)}\alpha_j
\left(h^0(E(n))\rk(E_j)-\rk(E)h^0(E_j(n))\right)\right.\\
   & \phantom{zp(n)[} +\delta_1\sum_{j=1}^{\len(E_\bullet)}\alpha_j\left(r\nu_{h(j)}(h^{(1)})-a_1\rk(E_j)\right)\\
  & \left.\phantom{zp(n)[}  +\delta_2\sum_{j=1}^{
\len(E_\bullet)}\alpha_j\left(r\nu_{h(j)}(h^{(2)})-a_2\rk(E_j)\right)  \right]\,.
\end{align*}

Let now $j^{(m)}_l:=j(h^{(m)}_l)$, $m=1,2$, $l=1,\ldots,a_m$. Since $T_{1|Y^{\otimes
\tup{h}^{(1)}}}\neq 0$ and the bundle $F^{\otimes
\tup{h}^{(1)}}$ generated by $Y^{\otimes \tup{h}^{(1)}}$ is contained in $E^{\otimes \tup{j}^{(1)}}$ it is clear that $\phi_{|E^{\otimes
\tup{j}^{(1)}}}\neq 0$. Thus, equation \eqref{eq:tmp_stab1} yields
\[
 \mu_1(E_\bullet,\tup{\alpha},\phi)\ge
\sum_{j=1}^{\len(E_\bullet)}\alpha_j\left(r\nu_{j}(\tup{j}^{(1)})-a_1\rk(E_j)\right)\,.
\]
Similarly, equation \eqref{eq:tmp_stab2} implies
\[
 \mu_2(E_\bullet,\tup{\alpha},s)\ge
\sum_{j=1}^{\len(E_\bullet)}\alpha_j\left(r\nu_{j}(\tup{j}^{(2)})-a_2\rk(E_j)\right)\,.
\]

Because $j(h)\le j$ is equivalent to $h\le h(j)$ we get
$\nu_{h(j)}(\tup{h}^{(m)})=\nu_j(\tup{j}^{(m)})$ for $m=1,2$. This finally leads to
\[
  \frac{\mu(\la,\gies_n(p))}{zp(n)}\le M^{\textnormal{s}}(E_\bullet,\tup{\alpha},n)
+\delta_1
\mu_1(E_\bullet,\tup{\alpha},\phi)+\delta_2\mu_2(E_\bullet,\tup{\alpha},s)\,.
\]
\end{proof}
%%%%%%%%%%%%%%%%%%%%%%%%%%

Let $(Y_\bullet,\tup{\beta})$ be a weighted flag of $Y$. Let $F_h\subset E$ be the subsheaf generated by $Y_h$ and let $E'_h$ be the subbundle generically generated by $F_h$. Let $E_\bullet$ be the flag induced by the subbundles $E'_h$ with $h^1(F_h(n))=0$. For an index $1\le j\le \len(E_\bullet)$ we set
$H(j):=\{1\le h\le \len(Y_\bullet)\,|\,E'_h=E_j\}$, $h(j):=\min H(j)$ and define
\[
\alpha_j:=\sum_{h\in H(j)}\beta_h\,,\qquad 1\le j< \len(E_\bullet)\,.
\]
This defines the weighted flag $Q_p(Y_\bullet,\tup{\beta}):=(E_\bullet,\tup{\alpha})$.

\begin{prop} \label{prop:stable=>GIT-stable}
 Suppose that $\delta_2a_2<1$, $n\ge n_3$ and $(E,L\phi,s)$ is $(\delta_1,\delta_2)$-semi-stable. 
 \begin{enumerate}
  \item For any one-parameter subgroup $\la$ with associated weighted flag $(Y_\bullet,\tup{\beta})$ the weighted flag $(E_\bullet,\tup{\alpha})=Q_p(Y_\bullet,\tup{\beta})$ satisfies 
  \[  
  \frac{\mu(\la,\gies_n(t))}{zp(n)} \ge 
 M^{\textnormal{s}}(E_\bullet,\tup{\alpha},n)
 +\delta_1\mu_1(E_\bullet,\tup{\alpha},\phi)+\delta_2(E_\bullet,\tup{\alpha},s)\,.
 \]
  \item If equality holds, we have $\len(Y_\bullet)=\len(E_\bullet)$, $h^1(F_j(n))=0$, $E_j$ is generated by $Y_j$ and $Y_j=H^0(q(n))^{-1}(H^0(E_j))$ for $1\le j \le \len(Y_\bullet)$.
 \end{enumerate} 
\end{prop}
\begin{proof}
 We  first assume $h^1(F_h(n))=0$ for all $h$. With $\dim(Y_h)\le h^0(F_h)$  and \eqref{eq:stab_in_Gies_0} we get
\begin{align*}
 \mu(\lambda,[M])=\, &\sum_{h=1}^{\len(Y_\bullet)}\beta_h \left(p(n)\rk(F_h)-r\dim(Y_h)\right)\\
  \ge\, & \sum_{j=1}^{\len(E_\bullet)}\sum_{h\in H(j)}\beta_h \left(p(n)\rk(F_h)-r
h^0(F_h(n))\right)\\
  & +\sum_{h\in H(\len(E_\bullet)+1)}\beta_hr \left(p(n)-\dim(Y_h)\right)\,.
\end{align*}
Note that here equality can only occur if $\dim(Y_h)=h^0(F(n))$ for all $h$.

Let $\tup{i}^{(m)}$, $m=1,2$, be tuples for which the minimum is attained in \eqref{eq:tmp_stab1} and \eqref{eq:tmp_stab2} respectively. We define $h^{(m)}_k:=h(i^{(m)}_k)$, $k=1,\ldots,a_m$, for $m=1,2$. Since $Y^{\otimes
\tup{h}^{(1)}}$ generically generates $E^{\otimes \tup{i}^{(1)}}$ and $\phi_{|E^{\otimes
\tup{i}^{(1)}}}\neq 0$ it follows that $T_{1|Y^{\otimes \tup{h}^{(1)}}}$ is non-trivial. Then, equation \eqref{eq:stab_in_Gies_1} says
\begin{align*}
 \mu(\lambda,[T_1])
  \ge\; & \sum_{j=1}^{\len(E_\bullet)}\sum_{h\in H(j)}\beta_h \left(p(n)\nu_h(\tup{h}^{(1)})-a_1h^0(F_h(n))\right)\\
  & +a_1\sum_{h\in H(\len(E_\bullet)+1)}\beta_h(p(n)-\dim(Y_h))\,.
\end{align*}

If the torsion sheaf $T_h:=E'_h/F_h$ does not vanish, $T_{2|Y^{\otimes \tup{h}^{(2)}}}$ may be trivial even if $s_{|E^{\otimes\tup{i}^{(2)}}}$ is not. Using \eqref{eq:stab_in_Gies_2} we merely find the estimate
\begin{align*}
 \mu(\lambda,[T_2])\ge\; & \sum_{j=1}^{\len(E_\bullet)}\sum_{h\in H(j)}\beta_h \left(p(n)(\nu_h(\tup{h}^{(2)})-a_2h^0(T_h))-a_2h^0(F_h(n))\right)\\
     &-a_2\sum_{h\in H(\len(E_\bullet)+1)}\beta_h \dim(Y_h)\,.
\end{align*}

The above calculations lead to
\begin{align*}
\frac{ \mu(\la,\gies_n(p))}{zp(n)} \ge& \sum_{j=1}^{\len(E_\bullet)}\sum_{h\in H(j)}\beta_h \left(h^0(E(n))\rk(F_h)-rh^0
(F_h(n))-a_2\delta_2
rh^0(T_h)\right) \\
  & + \delta_1\sum_{j=1}^{\len(E_\bullet)}\sum_{h\in H(j)}\beta_h \left(r\nu_h(\tup{h}^{(1)})-a_1\rk(F_h)\right) \\
  & + \delta_2\sum_{j=1}^{\len(E_\bullet)}\sum_{h\in H(j)}\beta_h \left(r \nu_h(\tup{h}^{(2)})-a_2\rk(F_h)\right)\\
  & +  r\sum_{h\in H(\len(E_\bullet)+1)} \beta_h \left(p(n)-\dim(Y_h)-a_2\delta_2 \right)\,.
\end{align*}
Due to our assumption on $a_2$ the term $(p(n)-\dim(Y_h)-a_2\delta_2)$ is positive.
Because we assume $h^1(F_h(n))=0$ we have $h^0(E'_h(n))=h^0(F_h(n))+h^0(T_h)$ and thus
\[
 rh^0(F_h(n))+a_2\delta_2 rh^0(T_h)= rh^0(E'_h(n))-rh^0(T_h)(1-a_2\delta_2)\le
rh^0(E'_h(n)).
\]
Note that equality can only hold if $T_h=0$ and thus $E'_h=F_h$ for all $h$.

Using $\rk(F_h)=\rk(E'_h)=\rk(E_j)$, $h^0(E'_h)=h^0(E_j)$ and
$\nu_h(\tup{h}^{(m)})=\nu_j(\tup{i}^{(m)})$ for $h\in H(j)$ one finds
\begin{align*}
\frac{\mu(\lambda,\gies_n(p))}{zp(n)} \ge
M^{\textnormal{s}}(E_\bullet,\tup{\alpha},n) +\delta_1\mu_1(E_\bullet,\tup{\alpha},\phi)+\delta_2\mu_2(E_\bullet,\tup{\alpha},s)\,.
\end{align*}

 Now let $(Y_\bullet,\tup{\beta})$ be an arbitrary weighted flag. Similarly to \autoref{def:decomposition} we decompose the flag $(Y_\bullet,\tup{\beta})$ into two flags
$(Y_\bullet^A,\tup{\beta}^A)$ and $(Y_\bullet^B,\tup{\beta}^B)$ such that $h^1(F^{B}_h(n))=0$ for all $1\le h\le \len(Y_\bullet^B)$ and $h^1(F^{A}_h(n))\neq 0$ for all $1\le h\le \len(Y_\bullet^A)$. Using \eqref{eq:estimate_V_abc} we get
\begin{align*}
\mu(\lambda,\gies(p)) =\;&  \eta
\mu(Y_\bullet,\tup{\beta},[M])+\theta_1\mu(\mu(Y_\bullet,\tup{\beta},[T_1]
)+\theta_2\mu(Y_\bullet,\tup{\beta},[T_2])\\
 \ge\; & 
\eta\mu(Y^A_\bullet,\tup{\beta}^A,[M])-(a_1\theta_1+a_2\theta_2)\sum_{h=1}^{\len(Y_\bullet^A)}
\beta^A_h\dim(Y^A_h)\\
  & + 
\eta\mu(Y^B_\bullet,\tup{\beta}^B,[M])+\theta_1\mu(Y^B_\bullet,\tup{\beta}^B,[T_1]
)+\theta_2\mu(Y^B_\bullet,\tup{\beta}^B,[T_2])\,.
\end{align*}

Using the linearization \eqref{eq:Linearisierung}, the stability condition \eqref{eq:stab_in_Gies_0} and $\dim(Y^A_h)\le h^0(F_h^A)$ we find
\begin{align*}
 &\eta\mu(Y^A_\bullet,\tup{\beta}^A,[M])-(a_1\theta_1+a_2\theta_2)\sum_{h=1}^{\len(Y^A_\bullet)}
\beta^A_h\dim(Y^A_h)\\
 \ge & zp(n)
\sum_{h=1}^{\len(Y^A_\bullet)}\beta^A_h\left[h^0(E(n))\rk(F^A_h)-h^0(F^A_h(n))\rk(E)-(a_1\delta_1+a_2\delta_2)\rk(F^A_h)\right].
\end{align*}
This expression is positive by \autoref{lem:h1=0=>positive}. Note that equality can only hold if $h^1(F_h(n))=0$ for all $h$.
\end{proof}

\begin{cor} \label{cor:GIT-stable<=>stable}
 For $a_2\delta_2<1$, $n\ge n_3$ and a point $p\in P_n$ the decorated swamp $(E,L,\phi,s)$ is $(\delta_1,\delta_2,n)$-section-(semi-)stable if and only if $\gies_n(p)$ is GIT (semi-)stable.
\end{cor}

%%%%%%%%%%%%%%%%%%%%%%%%%%%%%%%%%%%%%%%%%%%
\section{Proof of \autoref{main_thm}}
\label{sec:proof_of_main_thm}
This section is devoted to the proof of the main theorem. We first exhibit a family of decorated swamps with the local universal property. Then we check that our definition of S-equivalence in \autoref{def:S-equivalence} agrees with the general notion from \autoref{subsec:moduli_spaces}. Finally we show that the restriction of the Gieseker morphism to the locus of $(\delta_1,\delta_2)$-semi-stable decorated swamps is proper. Together with injectivity this will allow us to construct the good quotient of the parameter space. 

\subsection{The Local Universal Property}
We denote by $P^{\textnormal{(s)s}}_n:=\gies_n^{-1}(\Gies^{\textnormal{(s)s}})$ the open subscheme of $(\delta_1,\delta_2)$-(semi-)stable decorated swamps.
\begin{lemma}
 The family $(\tilde{E},\tilde{\ka},\tilde{N}_1,\tilde{\phi},\tilde{s})$ parameterized by $P^{\textnormal{(s)s}}_n$ satisfies the local universal property for $(\delta_1,\delta_2)$-\textup{(}semi\nobreakdash-\textup{)}stable decorated swamps.
\end{lemma}
\begin{proof}
 Let $D_S=(E_S,\ka_S,N_S,\phi_S,s_S)$ be a family parameterized by a scheme $S$ and $s\in S$ a point. By \autoref{prop:swamps_bounded} and the local universal property of the Quot scheme there is a neighborhood $U$ of $s$ and a morphism $f_0:U\to \Quot_n^0$ with $(f_0\times\id_X)^*Q\cong E_{|U\times X}$. The quotient $(f_0\times\id_X)^*q$ and the restriction of $D_S$ to $U$ define a family of decorated quotient swamps on $U$. By \autoref{prop:parameter_space} this defines a morphism $f:U\to P_n$ over $f_0$. By \autoref{cor:GIT-stable<=>stable} $f$ factorizes via $P_n^{\textnormal{(s)s}}$.
\end{proof}

\begin{prop}
 Let $f_1,f_2:S\to P_n$ be two morphisms. The pullbacks of the family $D:=(\tilde{E},\tilde{\ka},\tilde{N}_1,\tilde{\phi},\tilde{s})$ are isomorphic if and only if there exists a morphism $g:S\to \PGL(Y)$ such that $g\cdot f_1=f_2$.
\end{prop}
\begin{proof}
 Suppose $f_1^*D\cong f_2^*D$. Then there is a line bundle $L_S$ on $S$ and an isomorphism $(f_1\times\id_X)^*\tilde{E}\to (f_2\times\id_X)^*\tilde{E}\otimes L_S$. This leads to an isomorphism $Y\otimes\mc{O}_{S\times X}\to Y\otimes L_S$, inducing the desired morphism $g:G\to \PGL(Y)$. The converse is clear.
\end{proof}

\subsection{S-Equivalence}
Suppose $n\ge n_4$, $a_2\delta_2<1$ and let $p \in P^{\textnormal{ss}}_n$ be a point with associated decorated swamp $(E,L,\phi,s)$.
\begin{lemma}
 The maps $\Gamma_p$ and $Q_p$ define a bijection between the set of weighted flags $(Y_\bullet,\tup{\beta})$ of $Y$, such that any one-parameter subgroups $\la$ of $\SL(Y)$ with associated flag $(Y_\bullet,\tup{\beta})$ satisfies $\mu(\la,\gies_n(p))=0$ and the set of weighted flags $(E_\bullet,\tup{\alpha})$ of $E$ such that
 \begin{equation} \label{eq:M=0}
   M(E_\bullet,\tup{\alpha}) + \delta_1 \mu_1(E_\bullet,\tup{\alpha},\phi)
+ \delta_2\mu_2(E_\bullet,\tup{\alpha},s) = 0\,.
 \end{equation}
\end{lemma}
\begin{proof}
 If $\la$ is a one-parameter subgroup with associated flag $(Y_\bullet,\tup{\beta})$ and $\mu(\la,\gies_n(p))=0$ then by \autoref{prop:stable=>GIT-stable} the corresponding weighted flag $(E_\bullet,\tup{\alpha})=Q_p(Y_\bullet,\tup{\beta})$ satisfies \eqref{eq:M=0}. Again \autoref{prop:stable=>GIT-stable} shows $\Gamma_p\circ Q_p(Y_\bullet,\tup{\beta})=(Y_\bullet,\tup{\beta})$. 

Let $(E_\bullet,\tup{\alpha})$ be a weighted flag satisfying \eqref{eq:M=0}. By \autoref{prop:GIT-stable=>stable} we find $\mu(\la,\gies_n(p))=0$ for any one-parameter subgroup $\la$ with associated flag $(Y_\bullet,\tup{\beta})=\Gamma_p(E_\bullet,\tup{\alpha})$. Then $E_j$ is generically generated by $Y_j$. By \autoref{prop:stable=>GIT-stable} we get $Q_p(Y_\bullet,\tup{\beta})=(E_\bullet,\tup{\alpha})$.
\end{proof}

\begin{prop}
 Let $\la$ be a one-parameter subgroup of $\SL(Y)$ with $\mu(\la,\gies_n(p))=0$, $p_\infty:=\lim_{t\to \infty}\la(t)\cdot p$ the limit point and $(Y_\bullet,\tup{\beta})$ the associated weighted flag. Then, the admissible deformation of $(E,L,\phi,s)$ along $(E_\bullet,\tup{\alpha}):=Q_p(Y_\bullet,\tup{\beta})$ is isomorphic to the decorated swamp defined by $p_\infty$.
\end{prop}
\begin{proof}
 We choose a splitting $Y=\bigoplus_{i=1}^{\len(Y_\bullet)+1} Y^i$ such that $Y_j=\bigoplus_{i=1}^j Y^i$ for $j=1,\ldots,\len(Y_\bullet)$. Define $\tup{\ga}:=\tup{\ga}(Y_\bullet,\tup{\alpha})$ and let $\mc{Y}$ denote the sheaf   
  \[
  \mc{Y}:=\bigoplus_{i=1}^{\len(Y_\bullet)+1} Y^i \otimes t^{\ga_i-\ga_1 }\mc{O}_{S}=\sum_{j=1}^{\len(Y_\bullet)+1}Y_j\otimes t^{\ga_j-\ga_1}\mc{O}_{S}
 \]
 on $S:=\Spec(\CC[t])$. We consider the isomorphism $\psi:Y\otimes \mc{O}_S\to \mc{Y}$ given by $\psi(y\otimes 1)=y\otimes t^{\ga_i-\ga_1}$ for $y\in Y^i$ and the quotient
\[
 q_S: Y\otimes \pr_X^*\mc{O}_X(-n) \stackrel{\psi}{\to} \pr_S^*\mc{Y}\otimes \pr_X^*\mc{O}_X(-n)\to \sum_{j=1}^{\len(Y_\bullet)+1} t^{\ga_j-\ga_1}\pr_X^*E_j =: E_S
\]
on $S\times X$ induced by $q$.

Let $(F_\bullet,\tup{\alpha}^{(1)})$ and $(G_\bullet,\tup{\alpha}^{(2)})$ denote the induced weighted flags of $E_\rho$ and $E_\sigma$ respectively and set $\tup{\ga}^{(1)}:=\tup{\ga}(F_\bullet,\tup{\alpha}^{(1)})$,  $\tup{\ga}^{(2)}:=\tup{\ga}(G_\bullet,\tup{\alpha}^{(2)})$. Then the associated bundles of $E_S$ are 
\[
  E_{S,\rho}=\sum_{i=1}^{\len(F_\bullet)+1} t^{\ga^{(1)}_i-\ga^{(1)}_{1}}\pr_X^*F_i\,, \qquad 
  E_{S,\sigma}=\sum_{j=1}^{\len(G_\bullet)+1} t^{\ga^{(2)}_j-\ga^{(2)}_{1}}\pr_X^*G_j\,.
\]
Define $i_0:=\min\{i\,|\,\phi_{|F_i}\neq 0\}$ and $j_0:=\min\{j\,|\,s_{|G_j}\neq 0\}$ and consider the homomorphisms
\begin{align*}
 \phi_S: E_{S,\rho} \to t^{\ga^{(1)}_{i_0}-\ga^{(1)}_{1}}\pr_X^*L\,, \qquad
 s_S: E_{S,\sigma|S\times\{x_0\}} \to t^{\ga^{(2)}_{j_0}-\ga^{(2)}_{1}}\mc{O}_{S}\,
\end{align*}
induced by $\phi$ and $s$. Let $\ka_S:S\to \Jac^l$ be the constant map determined by $L$ and set $N_S:=t^{\ga'_{i_0}-\ga_1'}\mc{O}_S$. These data define a family of decorated quotient swamps $D:=(q_S,\ka_S,N_S,\phi,s)$. One checks that $D_{t=1}=(E,L,\phi,s)$ while $D_{t=0}=\opname{df}_{(E_\bullet,\tup{\alpha})}(E,L,\phi,s)$. The family $D$ also determines a morphism $f:S\to P_n$ such that $f(t)=\la(t)^{-1}\cdot p$. This proves the claim.
\end{proof}

\subsection{Properness of the Gieseker Morphism}
\begin{prop} \label{prop:properness}
Suppose $a_2\delta_2<1$. Then, there is an $n_4\ge n_3$ such that for $n\ge n_4$ the restriction
 \[
  \gies^{\textnormal{ss}}:=\gies_{n|P^{\textnormal{ss}}}:P^{\textnormal{ss}}_n \to \Gies_n^{\textnormal{ss}}
 \]
is proper.
\end{prop}
\begin{proof}
 We check the valuative criterion for discrete valuation rings: let $R$ be a discrete valuation ring, $Q$ its field of fractions and consider the diagram
 \[\xymatrix{
  \Spec(K) \ar[r]^-f \ar[d]_p & P^{\textnormal{ss}}_n \ar[d]^{\gies_n^{\textnormal{ss}}}\\
  \Spec(R) \ar[r]^-h &  \Gies_n^{\textnormal{ss}}\rlap{\,.}
  }
 \]
 The morphism $f$ defines a family of decorated quotient swamps $(q'_K,\kappa,N_1,N_2,\phi,s)$ on $\Spec(K)$. Since $\opname{Quot}$ is projective, the quotient $q'_K$ extends uniquely to $\Spec(R)$. The coherent sheaf $Q$ over the special point $0$ may have torsion, we consider the modified homomorphism
 \[
  q_R: Y\otimes \pr_X^*\mc{O}_X(-n) \to Q_R \to E_R:=Q_R^{\vee\vee}\,.
 \]
As a reflexive sheaf on a regular surface the sheaf $E_R$ is locally free (see \cite[1.3, 1.4]{Hart1980}). The composition $\ka_R:=\pr_{\Jac^d}\circ h$ is the unique extension of $\ka_K$ to $\Spec(R)$. Using $E_R$ and $\ka_R$ one constructs parameter space $\Phi$, projective over $\Spec(R)$, with a universal homomorphism $\tilde{\phi}:E_{R,\rho}\to \mc{O}_{\Phi}(1)\otimes (\ka_R\times\id_X)^*\mc{L}$. This allows to extend $N_1$ and $\phi_K$ to $R$. Finally, the homomorphism $s_K$ defines a morphism $\Spec(K)\to \Psi:=\PP(Y_{a_2,b_2})$, which has a unique extension $\Spec(R)\to \Psi$. This defines the family $(q_R,\ka_R,N_1,N_2,\phi_R,s_R)$ parameterized by $\Spec(R)$, which induces the morphism $h:\Spec(R)\to \Gies_n^{\textnormal{ss}}$. 

It remains to show that $q_R$ is a quotient. Let $q:Y\otimes \mc{O}_X(-n)\to E$ be the restriction of $q_R$ to the special fiber and let $U$ be the kernel of $H^0(q(n))$. GIT-semistability with respect to the weighted flag $(0\subset U\subset Y,(1))$ implies
\begin{equation} \label{eq:eigentlichkeit}
\begin{aligned}
 0  \le\, & \eta \mu(\lambda,[M])+\eta_1\mu(\lambda,[T_1])+\eta_2\mu(\lambda,[T_2])\\
   \le\, & zp(n)r(\delta_2a_2 -\dim(U))\,.
\end{aligned}
\end{equation}
Due to our assumption for $\delta_2$ this shows $U=\{0\}$.

Let now $E\to Q$ be a quotient of minimal slope and set $U:=\ker(Y\to H^0(Q(n)))$. GIT-semistability with respect to the flag $(\{0\}\subset U\subset Y,(1))$ yields
\begin{align*}
 0  \le & (p(n)-a_1\delta_1-a_2\delta_2)(p(n)\rk(F)-r\dim(U)) +(\delta_1ra_1+\delta_2ra_2)(p(n)-\dim(U))\\
 = & p(n)\big[(p(n)\rk(F)-r\dim(U))+(a_1\delta_1+a_2\delta_2)(r-\rk(F))\big]\,,
\end{align*}
where $F\subset E$ is the subsheaf generated by $U$. Using $\dim(U)\ge \dim(Y)-h^0(Q(n))$ and $\rk(E)\ge \rk(Q)+\rk(F)$ we find
\[
 \mu(E)+n+1-g=\frac{p(n)}{\rk(E)}\le
\frac{h^0(Q(n))}{\rk(Q)}+\frac{a_1\delta_1+a_2\delta_2}{\rk(E)}\,.
\]
Since $Q$ is semistable we have $h^0(Q(n))\le \rk(Q)[\mu(Q)+n+1]_+$ and thus
\[
 \mu_{\min}(E)=\mu(Q) \ge  \mu(E) -g -\frac{a_1\delta_1+a_2\delta_2}{\rk(E)}\,.
\]
This shows that the class of vector bundles $E$ arising from this construction is bounded. In particular there is an $n_4\ge n_3$, such that $h^1(E(n))=0$ and $E(n)$ is globally generated for $n\ge n_4$. Then, $h^0(E(n))=p(n)$ and $H^0(q(n))$ is an isomorphism. It follows that $q$ is surjective.

It is clear that $s_R$ factors via $E_{\sigma|\{x_0\}}$. Hence the family $(q_R,\ka_R,N_1,N_2,\phi_R,s_R)$ defines a morphism $\bar{f}:\Spec(R)\to P^{\textnormal{ss}}_n$ extending $f$. 
\end{proof}

\begin{proof}[Proof of \autoref{main_thm}]
 By \autoref{prop:Mumford} the projective good quotient $\Gies^{\textnormal{ss}}_n/\!\!/\SL(Y)$ and the geometric quotient $\Gies^{\textnormal{s}}_n/\SL(Y)$ exist. Since $\gies_n^{\textnormal{ss}}$ is proper and injective, hence quasi-finite, it is also finite by \cite[8.11.1]{EGAIV-3}. In particular, it is affine. Therefore, by \autoref{prop:affine_morphism->quotient} the projective good quotient $P_n^{\textnormal{ss}}/\!\!/\SL(Y)$ and the geometric quotient $P_n^{\textnormal{s}}/\SL(Y)$ also exist. These are also the quotients modulo $\PGL(Y)$, and since the projection $\SL(Y)\to \PGL(Y)$ has finite kernel, the notions of (semi-)stability with respect to these groups coincide. Now, \autoref{prop:good_quot_is_coarse_moduli} shows that these quotients are the coarse moduli spaces of $(\delta_1,\delta_2)$-(semi-)stable decorated swamps.
\end{proof}

%%%%%%%%%%%%%%%%%%%%%%%%%%%%%%%%%%
\section{Examples}
Immediate examples of our construction are the decorated vector bundles: A \emph{decorated vector bundle} is a pair $(E,s)$ consisting of vector bundle $E$ and a point $s\in \PP(E_{\sigma})_{|\{x_0\}}$. If we consider the trivial representation $\rho:\GL(r)\to \{1\}$, then the category of decorated swamps of type $(d,0)$ is equivalent to the category of decorated vector bundles of degree $d$. The function $\mu_1(E_\bullet,\tup{\alpha},\phi)$ is then always zero and one may effectively substitute $\delta_1=0$ in all calculations.

\subsection{Parabolic Vector Bundles}
Fix a sequence $0<r_1<\ldots,<r_{k+1}=r$ and let $F:=\Fl(\CC^r,\tup{r})$ denote the variety of flags of type $\tup{r}$. Let $\sigma$ be the natural action of $\GL(r)$ on $F$. A decorated vector bundle $(E,s)$ is then a quasi-parabolic vector bundle, i.e. a vector bundle $E$ with a flag $U_\bullet$ of type $\tup{r}$ in $E_{|\{x_0\}}$. The variety $F$ has a natural embedding into a product of Grassmannians. If we linearize the action in $\mc{O}_F(\beta_1,\ldots,\beta_k)$ we find that $(E,s)$ is  $\delta_2$-(semi-)stable if and only if
\[
 \frac{\opname{pardeg}(F)}{\rk(F)} (\le) \frac{\opname{pardeg}(E)}{\rk(E)} 
\]
holds for all subbundles $F$ with the \emph{parabolic degree} given by
\[
 \opname{pardeg}(F):= \deg(F) + \delta_2\sum_{i=1}^k \beta_i \dim( F_{|\{x_0\}}\cap U_i)\,.
\]
With the \emph{parabolic weights} $\tilde{\alpha}_i:=\delta_2 \sum_{j=i}^k\beta_j$ we recover the known stability condition from Mehta--Seshadri \cite{mehta80}.

\begin{rem}
 For the construction of the moduli space we use the embedding $F\to \PP(V_2)$ with
 \[
  V_2:=\bigotimes_{i=j}^k \left(\bigwedge^{r-r_j}V\right)^{\otimes \beta_j}\,.
 \]
This representation is polynomial and homogeneous of degree $a_2=\sum_{j=1}^k \beta_j(r-r_j)$. Due to the condition $a_2\delta_2<1$ we can only construct the moduli space for small weights. However, if in the construction of the Gieseker space we replace $\PP(Y_{a_2,b_2})$ by $\prod_{i=1}^k\opname{Gr}(Y,r-r_j)$ then the estimate \eqref{eq:estimate_mu(T_2)} can be improved to
\[
  -\sum_{j=1}^{\len(Y_\bullet)} \alpha_j\dim(Y_j) \sum_{i=1}^l\beta_i(r-r_i) \le \mu(\la,[T_2])\le \sum_{j=1}^{\len(Y_\bullet)} \alpha_j\dim(Y_j) \sum_{i=1}^l \beta_i \left(p(n)-(r-r_i) \right)\,.
\]
Using this one can relax the condition on $\delta_1$ in \autoref{prop:stable=>GIT-stable} and \autoref{prop:properness} to $\tilde{\alpha}_1<1$. We can therefore construct the moduli space for all admissible weights $\tilde{\alpha}_k<\ldots<\tilde{\alpha}_1<1$.
\end{rem}

\subsection{Vector Bundles with Level Structure}
Let $W$ be a complex vector space of dimension $r$. A level structure on a vector bundle $E$ is supposed to describe a trivialization $f:E_{|\{x_0\}}\to W$. In order to obtain a projective moduli space one needs to enlarge the category. 

Seshadri \cite{seshadri1982} introduced the \emph{level structure} as a homomorphism $f:E_{|\{x_0\}}\to W$. Two vector bundles with a level structure $(E,f)$ and $(E',f')$ are said to be \emph{isomorphic}, if there is an isomorphism $\psi:E\to E'$ such that $f'\circ \psi_{|\{x_0\}}=f$. If $\sigma$ is the natural action of $\GL(r)$ on $\PP(\Hom(\CC^r,W)^\vee)$ then the category of vector bundles with a level structure is equivalent to the category of decorated vector bundles. From \autoref{def:stability} it follows, that such an object is $\delta_2$-(semi-)stable if and only if every subbundle $F$ of $E$ satisfies
\[
 (\deg(E)-\delta_2)\rk(F)-(\deg(F)-\delta_2 c(F,f))\rk(E) (\ge) 0\,.
\]
where
\[
 c(F,f):=\begin{cases}0  & f_{|F_{|\{x_0\}}}=0 \\
		     1 &  f_{|F_{|\{x_0\}}}\neq 0 
	\end{cases}\,.
\]
This is equivalent to the stability condition of Seshadri \cite[\S 4, I, D\'ef. 2]{seshadri1982}.

Another approach uses completed homomorphisms: Consider $H:=\prod_{i=1}^r \left(\End(\bigwedge^i \CC^r )\setminus \{0\}\right)$ and let $\Omega'\subset H\times \CC^{r-1}$ be defined by
\[
 \Omega':=\left\{ (\tup{f},\tup{l})\in H\times \CC^{r-1}\,\biggm|\,\forall 1\le i\le r\,: \wedge^if_1=\prod_{j=1}^{i-1}l_j^{i-j}f_i,\, l_i\neq 0\right\}\,.
\]
The closure $\Omega$ of $\Omega'$ is the \emph{space of completed homomorphisms}. A point $f\in \Omega$ is written as $f:\CC^r\Rightarrow \CC^r$.

For a sequence $r_1<\ldots<r_{k+1}=r$ let $\Omega_{\tup{r}}$ be the subscheme
\[
 \Omega_{\tup{r}}:=\left\{(\tup{f},\tup{l})\in \Omega\,\Bigm|\,\forall 1\le i\le r\,: l_i=0 \Leftrightarrow i\in\tup{r} \right\}\,.
\]
This defines a stratification
\[
 \Omega=\bigcup_{\tup{r}} \Omega_{\tup{r}}\,.
\]
\begin{prop}[{\cite[Proposition 1]{lafforgue1998}}]
 A point $f\in \Omega_{\tup{r}}$ is uniquely determined by 
 \begin{enumerate}
 \item a tuple $(l_1,\ldots,l_{k})\in \CC^{r-1}$ with $l_i= 0 \Leftrightarrow i\in\tup{r}$ for $1\le i <r$,
  \item a descending flag $W_\bullet$ of length $k$ with $\dim(W_j)=r-r_j$, $1\le j \le k$,
  \item an ascending flag $W'_\bullet$ of length $k$ with $\dim(W'_j)=r_i$, $1\le j \le k$,
  \item  and isomorphisms $v_j:W_{j-1}/W_{j}\to W'_j/W'_{j-1}$ for $1\le j\le k+1$.
 \end{enumerate}
\end{prop}

For our purposes we only need to recall how one recovers the point $f\in \Omega_{\tup{r}}$ from the flags $W_\bullet$, $W'_\bullet$ and the isomorphisms $v_i$, $1\le i\le k+1$: For a number $i=1,\ldots,r$ we define
 \begin{align*}
 j_-(i):&=\max\{ j\,|\,j=0,\ldots,k\,,\, r_j<i\}\,, & i_-:&=r_{j_-(i)}\,,\\
 j^+(i):&=\min\{ j\,|\, j=0,\ldots,k\,,\, i\le r_j\}\,, & i^+:&=r_{j^+(i)} \,.
\end{align*}
The isomorphisms $v_j$ induce an morphism $h_i$
\begin{equation}
\begin{gathered}
\label{eq:hom}
  \xymatrix{
  {\displaystyle \bigwedge^i \CC^r}  \ar@{->>}[r] \ar@{-->}[d]_{h_i} & {\displaystyle \bigwedge^{i-i_-} (W_{j_-(i)}/W_{j^+(i)}) \otimes \bigotimes_{j=1}^{j_-(i)} \bigwedge^{r_j-r_{j-1}}(W_{j-1}/W_j)} \ar[d] \\
   {\displaystyle \bigwedge^i \CC^r} & 
  {\displaystyle \bigwedge^{i-i_-} (W'_{j^+(i)}/W'_{j_-(i)}) \otimes \bigotimes_{j=1}^{j_-(i)} \bigwedge^{r_j-r_{j-1}} (W'_{j}/W'_{j-1})} \ar@{>->}[l] \rlap{\,.}
  }
  \end{gathered}
\end{equation}
 If one defines
 \[
  f_i:= \left(\prod_{\substack{1\le j <i\\j\notin \tup{r} }} l_j^{i-j} \right) h_i\,,
 \]
then $f=(\tup{f},\tup{l})\in \Omega_{\tup{r}}$ is the desired point.

The group $(\CC^*)^r$ acts freely on $\Omega$ by
 \begin{align*}
  (\CC^*)^r \times \Omega & \to \Omega\\
   (\tup{z},(\tup{f},\tup{l}))&\mapsto (\tup{f}',\tup{l}')
 \end{align*}
 with $f'_i:=z_if_i$, $1\le i\le r$ and $l'_i:=z_{i-1}^{-1} z_i^2 z_{i+1}^{-1}l_i$ for $1\le i \le r-1$, where $z_0:=1$. The quotient $\overline{\PGL}(r,\CC):=\Omega/(\CC^*)^r$ is isomorphic to the closure of the $\GL(r,\CC)$-orbit of $([\id_{\bigwedge^i \CC^r}],1\le i \le r)$ in $\prod_{i=1}^r \PP(\End(\bigwedge^i \CC^r)^\vee)$. This is the \emph{wonderful compactification} of De Concini--Procesi \cite{DCP83}.
 
 A \emph{vector bundle with a level structure} is defined to be a vector bundle $E$ together with a complete homomorphism $f:E_{|\{x_0\}}\Rightarrow \CC^ r$.  Let $\sigma$ be the natural action of $\GL(r)$ on $\overline{\PGL}(r,\CC)$. Then the category of vector bundles with a level structure is equivalent to the category of decorated vector bundles. We linearize the action in the line bundle
  \[
 \prod_{i=1}^r \pr_{\PP\left(\End(\bigwedge^i \CC^r)^\vee \right)}^*\mc{O}_{\PP\left(\End(\bigwedge^i \CC^r)^\vee\right)}(\theta_i)\,.
\]
 
A level structure $f$ on a vector bundle $E$ determines a sequence $r_1<\ldots<r_k<r$, a descending flag $E_\bullet$ of $E_{|\{x_0\}}$, an ascending flag $W'_\bullet$ of $\CC^r$ and elements $[v_i]\in \opname{Iso}(W_{i-1}/W_{i}, W'_i/W'_{i-1})/\CC^*$.
 
For a subspace $V\subset E_{|\{x_0\}}$ we define
\begin{equation*}
 c_i(V,W_\bullet) :=\min\left\{\dim(V/(V\cap W_{j^+(i)})), \dim(V/(V\cap W_{j_-(i)}))+ i-i_-\right\}\,.
\end{equation*}
 \begin{prop} \label{prop:stability_of_VB_w_level_structure}
 A vector bundle with a level structure $(E,f)$ is $\delta_2$-\textup{(}semi\nobreakdash-\textup{)}stable if and only if any non-trivial proper subbundle $F\subset E$ satisfies
 \[
  \deg_{\tup{\theta}}(E)\rk(F)-\deg_{\tup{\theta}}(F)\rk(E) (\ge) 0
 \]
where
\[
 \deg_{\tup{\theta}}(F):=\deg(F)-\delta_2\sum_{i=1}^r\theta_i \,c_i(F_{|\{x_0\}},W_\bullet) \,.
\]
\end{prop}
\begin{proof}
Let $V:=E_{|\{x_0\}}$ and define $h_i:\bigwedge^i V\to \bigwedge^i \CC^r$ by the diagram similar to \eqref{eq:hom}.
 Then for the point $[h_i]\in \PP(\Hom(\bigwedge^i V,\bigwedge^i \CC^r)^\vee)$ and a one-parameter subgroup $\la$ of $\SL(V)$ with weighted flag $(V_\bullet,\tup{\alpha})$ one calculates
 \begin{align*}
 \mu(\la,[u_i]) =  \sum_{k=1}^{\len(V_\bullet)}\alpha_k \big(r c_i(V_k,W_\bullet) -\dim(V_k) i \big)\,.
\end{align*}
From this one easily deduces the claim.
\end{proof}

In order to compare this to results in the literature we introduce the polynomial
\[
 q(s):= \sum_{1\le i\le s} \theta_i i + \sum_{s<i\le r} \theta_i s -\frac{s}{r}\sum_{i=1}^r \theta_i i\,.
\]
\begin{lemma}
We have
\[
  \sum_{i=1}^r \theta_i \left(r c_i(F_{|\{x_0\}},W_\bullet)-\rk(F)i\right)
  = r\sum_{s\in \tup{r}}\left(q\left(s_- +\dim(\overline{W}_{j_-(s)})/\overline{W}_{j^+(s)})\right)-q(s_-)\right)\,,
\]
where  $\overline{W}_j:=F_{|\{x_0\}}\cap W_j$.
\end{lemma}
\begin{proof}
We calculate
 \begin{align*}
  A:=\,&\sum_{s\in \tup{r}} (q(s_- +\dim(\overline{W}_{j_-(s)}/\overline{W}_{j^+(s)}))-q(s_-))\\
  =\,&  \sum_{i=1}^r \theta_i\left(\sum_{\substack{s\in \tup{r}\\s> i^+}}(i-i)+B_{s=i^+} + \sum_{\substack{s\in \tup{r}\\s\le i_-}}\dim(\overline{W}_{j_-(s)}/\overline{W}_{j^+(s)}) \right) -\left(\sum_{i=1}^r \theta_i i\right)\frac{\rk(F)}{r}\,,
 \end{align*}
where
\begin{align*}
 B_{s=i^+}=\,&\begin{cases}
    \dim(\overline{W}_{j_-(s)}/\overline{W}_{j^+(s)}) \,, & s_- + \dim(\overline{W}_{j_-(s)}/\overline{W}_{j^+(s)})< i\,,\\
    t - s_- \,, & s_- +\dim(\overline{W}_{j_-(s)}/\overline{W}_{j^+(s)}) \ge i
   \end{cases}\\
   =\,& \min\left\{\dim(\overline{W}_{j_-(s)}/\overline{W}_{j^+(s)}), i-i_-\right\}\,.
\end{align*}
Due to
\[
\sum_{\substack{s\in \tup{r}\\s\le i_-}}\dim(\overline{W}_{j_-(s)}/\overline{W}_{j^+(s)})=\dim(F_{|\{x_0\}}/\overline{W}_{j_-(i)})
\]
one finally gets
\[
 A=  \sum_{i=1}^r \theta_i c_i(F_{|\{x_0\}},W_\bullet)  -\left(\sum_{i=1}^r \theta_i i \rk(F)\right)  \frac{1}{r}\,.
\]
\end{proof}

\begin{cor}
 A vector bundle with a level structure $(E,f)$ is $\delta_2$-\textup{(}semi\nobreakdash-\textup{)}stable if and only if any non-trivial proper subbundle $F\subset E$ satisfies
 \[
  \deg(F)r (\le) \deg(E)\rk(F)+r\delta_2\sum_{s\in\tup{r}}q\left(s_-+\dim(\overline{W}_{j_-(s)}/\overline{W}_{j^+(s)})\right) -q(s_-)\,.
\]
\end{cor}
Thus, we recover the stability condition that Ng\^o Dac used for the compactification of the stack of shtukas \cite[Th\'eor\`eme A]{Ngo2007}.

\subsection{Parabolic Higgs Bundles and Decorated principal bundles}
As mentioned in the abstract parabolic Higgs bundles can also be treated as a special case of decorated swamps. However, the stability concept does not immediately reduce to the known definition of stability. Instead one recovers the known stability concept in the limit of large $\delta_1$. 

Following the strategy of Schmitt \cite{schmitt08} one can use the moduli space of decorated swamps to construct the moduli space of decorated principal bundles. Examples of these objects include parabolic principal bundles and principal bundles with a level structure.

These results are contained in the authors thesis \cite{Beck2014} and will appear in a seperate publication.

\end{document}